\documentclass[letterpaper]{amsart}
\pdfoutput=1
\newtheorem{theorem}{Theorem}[section]

\newtheorem{proposition}[theorem]{Proposition}
\newtheorem{corollary}[theorem]{Corollary}

\theoremstyle{definition}
\newtheorem{definition}[theorem]{Definition}

\theoremstyle{remark}

\usepackage{hyperref}
\usepackage{graphicx}
\usepackage{color}
\usepackage{amsmath}
\usepackage{amsfonts}
\usepackage{euler}
\usepackage{amssymb}
\usepackage{epsfig}
\usepackage{rotating}
\usepackage{graphics}
\newcommand{\be}{\begin{equation}}

\newcommand{\ee}{\end{equation}}

\definecolor{blue}{cmyk}{1.,1.,0.,0.63}
\definecolor{red}{cmyk}{0.,1.,1.,0.63}
\definecolor{green}{cmyk}{1.,0.,1.,0.63}
\definecolor{black}{cmyk}{1.,1.,1.,1.}

\newcommand{\blue}{\textcolor{blue}}

\newcommand{\red}{\textcolor{red}}

\newcommand{\Om}{\Omega}

\newcommand{\K}{\Upsilon}

\newcommand{\om}{\omega}

\newcommand{\f}{\varphi}

\newcommand{\eps}{\varepsilon}

\newcommand{\dz}{\wedge}

\newcommand{\C}{{\bf C }}

\newcommand{\R}{{\bf R }}

\renewcommand{\K}{{\bf K }}

\newcommand{\ba}{\begin{array}}

\newcommand{\ea}{\end{array}}

\newcommand{\beq}{\begin{eqnarray}}

\newcommand{\eeq}{\end{eqnarray}}

\newtheorem{lm}{lemma}

\newtheorem{thee}{theorem}

\newtheorem{proo}{proposition}

\newtheorem{co}{corollary}

\newtheorem{rem}{remark}

\newtheorem{deff}{definition}

\newcommand{\bd}{\begin{deff}}

\newcommand{\ed}{\end{deff}}

\newcommand{\bl}{\begin{lm}}

\newcommand{\el}{\end{lm}}

\newcommand{\bp}{\begin{proo}}

\newcommand{\ep}{\end{proo}}

\newcommand{\bt}{\begin{thee}}

\newcommand{\et}{\end{thee}}

\newcommand{\bc}{\begin{co}}

\newcommand{\ec}{\end{co}}

\newcommand{\brm}{\begin{rem}}

\newcommand{\erm}{\end{rem}}

\newcommand{\der}{{\rm d}}

\newcommand{\sgn}{\mathrm{sgn}}

\usepackage{t1enc}
\def\frak{\mathfrak}

\newcommand{\newc}{\newcommand}

\renewcommand{\exp}{\operatorname{exp}}
\newcommand{\id}{\operatorname{id}}

\renewcommand{\Re}{\operatorname{Re}}

\let\ccdot\cdot
\def\cdot{\hbox to 2.5pt{\hss$\ccdot$\hss}}

\newc{\aR}{\mbox{\boldmath{$ R$}}}
\newc{\aS}{\mbox{\boldmath{$ S$}}}
\newc{\aT}{\mbox{\boldmath{$ T$}}}
\newc{\aW}{\mbox{\boldmath{$ W$}}}

\newc{\aK}{\mbox{\boldmath{$ K$}}}
\newc{\aL}{\mbox{\boldmath{$ L$}}}

\newcommand{\bbC}{\mathbb{C}}

\usepackage{amssymb}
\usepackage{amscd}

\let\f=\varphi

\newcommand{\bma}{\begin{pmatrix}}
\newcommand{\ema}{\end{pmatrix}}

\newc{\obstrn}[2]{B^{#1}_{#2}}

\newcommand{\rpl}                         
{\mbox{$
\begin{picture}(12.7,8)(-.5,-1)
\put(0,0.2){$+$}
\put(4.2,2.8){\oval(8,8)[r]}
\end{picture}$}}

\newcommand{\lpl}                         
{\mbox{$
\begin{picture}(12.7,8)(-.5,-1)
\put(2,0.2){$+$}
\put(6.2,2.8){\oval(8,8)[l]}
\end{picture}$}}

\usepackage{ifthen}

\newc{\tensor}[1]{#1}
\newc{\Mvariable}[1]{\mbox{#1}}
\newc{\down}[1]{{}_{#1}}
\newc{\up}[1]{{}^{#1}}

\newc{\JulyStrut}{\rule{0mm}{6mm}}
\newc{\midtenPan}{\mbox{\sf S}}
\newc{\midten}{\mbox{\sf T}}
\newc{\midtenEi}{\mbox{\sf U}}
\newc{\ATen}{\mbox{\sf E}}
\newc{\BTen}{\mbox{\sf F}}
\newc{\CTen}{\mbox{\sf G}}

\def\sideremark#1{\ifvmode\leavevmode\fi\vadjust{\vbox to0pt{\vss
 \hbox to 0pt{\hskip\hsize\hskip1em
 \vbox{\hsize3cm\tiny\raggedright\pretolerance10000
 \noindent #1\hfill}\hss}\vbox to8pt{\vfil}\vss}}}%

\numberwithin{equation}{section}

\newcounter{romenumi}
\newcommand{\labelromenumi}{(\roman{romenumi})}

\begin{document}

\title[On degenerate para-CR structures]{
On degenerate para-CR structures:
\\
Cartan reduction and homogeneous models}


\vskip 1.truecm
\author{Jo\"el Merker} \address{Laboratoire de Math\'ematiques d'Orsay,
CNRS, Universit\'e Paris-Saclay, 91405 Orsay Cedex,
France}

\email{joel.merker@universite-paris-saclay.fr}
\author{Pawe\l~ Nurowski} \address{Centrum Fizyki Teoretycznej,
Polska Akademia Nauk, Al. Lotnik\'ow 32/46, 02-668 Warszawa, Poland}
\email{nurowski@cft.edu.pl}

\thanks{
\\
${}\ \ \ \ \ $
This work was supported
in part by the Polish National Science Centre (NCN) 
via the grant number 2018/29/B/ST1/02583.}


\begin{abstract}
Motivated by recent works in Levi degenerate CR geometry, 
this article endeavours to study the wider and more flexible 
{\sl para-CR structures} for which the constraint of 
invariancy under complex conjugation is relaxed.
We consider $5$-dimensional para-CR structures 
whose Levi forms are of constant rank $1$ 
and that are $2$-nondegenerate both with respect to
parameters and to variables. 
Eliminating parameters, such structures may be represented
modulo point transformations
by pairs of PDEs $z_y = F(x, y, z, z_x)$ $\,\,\&\,\,$ 
$z_{xxx} = H(x, y, z, z_x, z_{xx})$, 
with $F$ independent of $z_{xx}$ and $F_{z_xz_x} \neq 0$,
that are completely integrable $D_x^3 F = \Delta_y H$,

Performing at an advanced level Cartan's method of equivalence,
we determine {\em all} concerned homogeneous models,
together with their symmetries:

\smallskip\noindent
{\bf (i)}\,\, 
$z_y = \tfrac14 (z_x)^2\quad \&\quad z_{xxx}=0$;

\smallskip\noindent
{\bf (ii)}\,\, 
$z_y = \tfrac14 (z_x)^2\quad \& \quad z_{xxx}=(z_{xx})^3$;

\smallskip\noindent
{\bf (iiia)}\,\, 
$z_y = \tfrac14 (z_x)^b\,\, \& \,\,z_{xxx} = 
(2-b)\frac{(z_{xx})^2}{z_x}$ with $z_x > 0$ 
for any real $b \in [1, 2)$;

\smallskip\noindent
{\bf (iiib)}\,\, 
$z_y = f(z_x)\quad \& \quad z_{xxx} = h(z_x)\big(z_{xx}\big)^2$, 
where the function $f$ is determined by the implicit equation:
\[
\big(z_x^2+f(z_x)^2\big)
\mathrm{exp}
\Big(
2b\,\mathrm{arctan}\tfrac{bz_x-f(z_x)}{z_x+bf(z_x)}
\Big)
\,=\,
1+b^2
\] 
and where, for any real $b > 0$:
\[
h(z_x)
\,:=\,
\frac{(b^2-3)z_x-4bf(z_x)}{\big(f(z_x)-bz_x\big)^2}.
\]
\end{abstract}
\maketitle

\vspace{-1truecm}
\tableofcontents
\newcommand{\bbS}{\mathbb{S}}
\newcommand{\bbR}{\mathbb{R}}
\newcommand{\sog}{\mathbf{SO}}
\newcommand{\glg}{\mathbf{GL}}
\newcommand{\slg}{\mathbf{SL}}
\newcommand{\og}{\mathbf{O}}
\newcommand{\soa}{\frak{so}}
\newcommand{\gla}{\frak{gl}}
\newcommand{\sla}{\frak{sl}}
\newcommand{\sua}{\frak{su}}
\newcommand{\dr}{\mathrm{d}}
\newcommand{\sug}{\mathbf{SU}}
\newcommand{\cspg}{\mathbf{CSp}}
\newcommand{\gat}{\tilde{\gamma}}
\newcommand{\Gat}{\tilde{\Gamma}}
\newcommand{\thet}{\tilde{\theta}}
\newcommand{\Thet}{\tilde{T}}
\newcommand{\rt}{\tilde{r}}
\newcommand{\st}{\sqrt{3}}
\newcommand{\kat}{\tilde{\kappa}}
\newcommand{\kz}{{K^{{~}^{\hskip-3.1mm\circ}}}}
\newcommand{\bv}{{\bf v}}
\newcommand{\di}{{\rm div}}
\newcommand{\curl}{{\rm curl}}
\newcommand{\cs}{(M,{\rm T}^{1,0})}
\newcommand{\tn}{{\mathcal N}}
\newcommand{\ten}{{\Upsilon}}

\section{Introduction}
\label{introduction}

In~{\cite{Nurowski-Sparling-2003}},
the second-named author and Sparling explored in
depth the close relationships between the geometry associated with
second order ordinary differential equations defined modulo point
transformations of variables, and the geometry of three-dimensional
Cauchy-Riemann (CR) structures, 
{\em cf.} also~{\cite{Nurowski-Tafel-1988, Nurowski-2005,
Godlinski-Nurowski-2009}}.
The goal of this article is to
explain how certain {\em degenerate} five-dimensional CR structures
give rise, analogously, to certain closely tied pairs of PDEs, 
and then, to find all the concerned homogeneous geometries,
by employing Cartan's method of equivalence.

Using a purely Lie-theoretical method, 
in their 2008 extensive {\em Acta
Mathematica} paper~{\cite{Fels-Kaup-2008}}, Fels-Kaup classified all
homogeneous $2$-nondegenerate constant Levi rank $1$ hypersurfaces
$M^5 \subset \C^3$.  Such hypersurfaces are termed `{\sl
$\mathfrak{C}_{2,1}$ hypersurfaces}', and our pairs of PDEs in
question are issued from them by 
parameters elimination ({\em see} below), after
complexifying and relaxing the invariancy under complex conjugation.

A decade ago, no Cartan-type reduction to an $\{e\}$-structure
bundle was available for $\mathfrak{C}_{2,1}$ hypersurfaces.  Since
then, the Cartan(-Tanaka) method was by applied by
Medori-Spiro~{\cite{Medori-Spiro-2014}},
and in a parametric way by
Pocchiola, Foo and the first-named
author~{\cite{Merker-Pocchiola-2018, Foo-Merker-2019}}, who found two
primary (relative) differential invariants $W_0$ and $J_0$.
The identical vanishing $W_0(M) \equiv 0 \equiv J_0(M)$
characterizes {\sl flatness}, namely biholomorphic equivalence of $M$
to the flat model which is graphed in $\C^3 \ni (z, \zeta, w)$ as
$\Re\, w = \big( z \overline{z} + \frac{1}{2} z^2\overline{\zeta} +
\frac{1}{2} \overline{z}^2 \zeta \big) \big/ \big( 1 - \zeta
\overline{\zeta} \big)$, which was set up by the firt-named author and
Gaussier~{\cite{Gaussier-Merker-2003}}, and which was shown by
Fels-Kaup~{\cite{Fels-Kaup-2007}} to be locally biholomorphic 
to the tube $S^2 \times i\R^3 \subset \C^3$ over the
future light cone $S^2 := \big\{ x\in\R^3 \colon\, x_1^2+x_2^2 =
x_3^2,\,\, x_3>0 \big\}$. Two recent 
prepublications~{\cite{Chen-Foo-Merker-Ta-2019, Foo-Merker-Ta-2019}}
construct Poincar\'e-Moser normal forms
for $\mathfrak{C}_{2,1}$ hypersurfaces.

Because a forthcoming survey~{\cite{Merker-Nurowski-2020}}
will expose more complete historical and synthetic aspects,
we now directly come to the heart of the matter,
{\em i.e.} we start by presenting the PDE systems 
studied in this article. Then we perform a
precise description of the contents of our contribution,
relating it to CR and affine geometry.

Given a $\mathcal{C}^\omega$ real hypersurface $M^5 \subset \C^3$ of
complex-graphed equation: 
\[
w
\,=\, 
Q(z_1, z_2, \overline{z}_1,\overline{z}_2,\overline{w})
\]
obtained by solving for $w$ a real
implicit equation $\rho(z_1, z_2, w, \overline{z}_1, \overline{z}_2,
\overline{w}) = 0$, one can forget about complex conjugation, work
over the field $\K = \R$ or $\K = \C$, and consider instead, in
coordinates $(x,y,z, a,b,c)$ a so-called {\em submanifold of
solutions} $\mathcal{M} \subset \K_{x,y,z}^{2+1} \times
\K_{a,b,c}^{2+1}$ having two equivalent equations:
\[
z
\,=\,
Q(x,y,a,b,c)
\ \ \ \ \ \ \ \ \ \ \ \ \ \ \ \ \ \ \ \
\text{and}
\ \ \ \ \ \ \ \ \ \ \ \ \ \ \ \ \ \ \ \
c
\,=\,
P(a,b,x,y,z).
\]
One thinks that $(x,y,z)$ are the {\sl variables},
while $(a,b,c)$ are the {\sl parameters}.
Two Levi forms, with respect to parameters
and with respect to variables, can be defined
They are represented by two $2 \times 2$ matrices:
\[
\left(\!
\begin{array}{cc}
\frac{-Q_cQ_{xa}+Q_aQ_{xc}}{Q_c^2} &
\frac{-Q_cQ_{xb}+Q_bQ_{xc}}{Q_c^2}
\\
\frac{-Q_cQ_{ya}+Q_aQ_{yc}}{Q_c^2} &
\frac{-Q_cQ_{yb}+Q_bQ_{yc}}{Q_c^2}
\end{array}
\!\right)
\ \ \ \ \ 
\text{and}
\ \ \ \ \ 
\left(\!
\begin{array}{cc}
\frac{-P_zP_{ax}+P_xP_{az}}{P_z^2} &
\frac{-P_zP_{ay}+P_yP_{az}}{P_z^2}
\\
\frac{-P_zP_{bx}+P_xP_{bz}}{P_z^2} &
\frac{-P_zP_{by}+P_yP_{bz}}{P_z^2}
\end{array}
\!\right).
\]
Furthermore, these two Levi forms are linked in a 
way~{\cite[Lm.~9.1]{Merker-Nurowski-2018}} that guarantees:
\[
{\rm rank}\,
{\rm Levi}_{\rm par}
(Q)
\,=\,
{\rm rank}\,
{\rm Levi}_{\rm var}
(P).
\]
As in~{\cite{Merker-Pocchiola-2018}}, 
we will assume that the Levi forms have
(common) constant rank $1$.

Also, similarly as for CR manifolds, two 
{\em nonequivalent} notions of $2$-nondegeneracy,
with respect to parameters and to variables, may be 
defined~{\cite[Sections~15, 20]{Merker-Nurowski-2018}}.
They are expressed invariantly by:
\[
0
\,\neq\,
\left\vert\!
\begin{array}{ccc}
Q_a & Q_b & Q_c
\\
Q_{xa} & Q_{xb} & Q_{xc}
\\
Q_{xxa} & Q_{xxb} & Q_{xxc}
\end{array}
\!\right\vert
\ \ \ \ \ \ \ \ \ \ \ \ \ \ \ \ \ \ \ \
\text{and}
\ \ \ \ \ \ \ \ \ \ \ \ \ \ \ \ \ \ \ \
0
\,\neq\,
\left\vert\!
\begin{array}{ccc}
P_x & P_y & P_z
\\
P_{ax} & P_{ay} & P_{az}
\\
P_{aax} & P_{aay} & P_{aaz}
\end{array}
\!\right\vert.
\]

As Segre did in~{\cite{Segre-1931-a}},
from the three equations:
\[
z
\,=\,
Q(x,y,a,b,c),
\ \ \ \ \ \ \ \ \ \
z_x
\,=\,
Q_x(x,y,a,b,c),
\ \ \ \ \ \ \ \ \ \
z_{xx}
\,=\,
Q_{xx}(x,y,a,b,c),
\]
assuming $2$-nondegeneracy with respect to parameters,
we can solve $(a,b,c)$ and replace them in $z_y = Q_y$,
$z_{xxx} = Q_{xxx}$, obtaining a completely integrable
system of two PDEs:
\begin{equation}
\label{two-PDEs}
z_y
\,=\,
F(x,y,z,z_x,z_{xx})
\ \ \ \ \
\&
\ \ \ \ \
z_{xxx}
\,=\,
H(x,y,z,z_x,z_{xx}).
\end{equation}
It is elementary to verify~{\cite[Prp.~23.1]{Merker-Nurowski-2018}
that the rank of the
Levi form of the submanifold of solutions was $1$
if and only if:
\[
0
\,\equiv\,
F_{z_{xx}}.
\]
So we do assume that $F$ is independent of $z_{xx}$.
It is also elementary to 
verify~{\cite[Prp.~26.2]{Merker-Nurowski-2018}
that the submanifold of solutions was $2$-nondegenerate
with respect to {\em variables} if and only if:
\[
0
\,\neq\,
F_{z_xz_x}.
\]
The degenerate branch $F_{z_xz_x} \equiv 0$ will not
be studied in this article, and we will constantly
assume $F_{z_{xx}} \equiv 0 \neq F_{z_xz_x}$.

The graphed model
inspired from~{\cite{Gaussier-Merker-2003}},
rewritten $z + c = \frac{2xa + x^2b + a^2y}{1-yb}$,
conducts to the model PDE system:
\[
z_y
\,=\,
\tfrac{1}{4}\,
(z_x)^2
\ \ \ \ \
\&
\ \ \ \ \
z_{xxx}
\,=\,
0.
\]
In Subsection~{\ref{flatm}}, we show that its Lie group of
decoupled symmetries~{\cite{Merker-2008, Hill-Nurowski-2010}}:
\[
(x,y,z,\,a,b,c)
\,\longmapsto\,\,
\big((x'(x,y,z),y'(x,y,z),z'(x,y,z),\,
a'(a,b,c),b'(a,b,c),c'(a,b,c)\big),
\]
which are point equivalences of the PDE system,
is isomorphic to $\sog(3,2)$.

Passing to the general case, 
introducing the two total differentiation operators pulled-back
to the PDE system:
\[
D
\,:=\,
\partial_x
+
p\,\partial_z
+
r\,\partial_p
+
H\,\partial_r
\ \ \ \ \
\&
\ \ \ \ \
\Delta
\,:=\,
\partial_y
+
F\,\partial_z
+
DF\,\partial_p
+
D^2F\,\partial_r,
\]
the complete integrability
expresses as $D^3F = \Delta H$, and 
guarantees~{\cite[\S~1]{Merker-2008}} that the general
solution is of the form $Q(x,y,a,b,c)$. 

Forgetting about submanifolds of solutions, we launch
Cartan's method
by defining a $2$-nondegenerate para-CR 
structure on a real $5$-manifold $M \ni (x,y,z,p,r)$ associated 
with the above two PDEs~({\ref{two-PDEs}})
as an equivalence class of $1$-forms
modulo point equivalences in terms of 
an {\sl initial} coframe of (contact) $1$-forms,
together with {\sl lifted $1$-forms}, `rotated' by
an initial $G$-structure:
\[
\aligned
\om^1
&
\,:=\,
\der z-p\der x-F\der y,
\\
\om^2
&
\,:=\,
\der p-r\der x-DF\der y,
\\
\om^3
&
\,:=\,
\der r-H\der x-D^2F\der y,
\\
\om^4
&
\,:=\,
\der x,
\\
\om^5
&
\,:=\,
\der y,
\endaligned
\ \ \ \ \ \ \ \ \ \ \ 
\bma
\theta^1\\\theta^2\\\theta^3\\\theta^4\\\theta^5\ema 
\,:=\,
\bma f^1&0&0&0&0
\\ 
f^2&\rho{\rm e}^{\phi}&f^4&0&0
\\ 
f^5&f^6&f^7&0&0
\\ 
\bar{f}{}^2&0&0&\rho{\rm e}^{-\phi}&\bar{f}{}^4
\\
\bar{f}{}^5&0&0&\bar{f}{}^6&\bar{f}{}^7
\ema
\bma
\om^1
\\
\om^2
\\
\om^3
\\
\om^4
\\
\om^5
\ema.
\]
Similarly to the CR case~({\cite{Merker-Pocchiola-2018,
Foo-Merker-2019}}), we perform several
torsion normalizations, which lead us to change the initial coframe
on $M$ into:
\[
\bma
\om^1
\\
\om^2
\\
\om^3
\\
\om^4
\\
\om^5
\ema
\,\,\,\longmapsto\,\,\,
\bma 
-1&0&0&0&0\\0&1&0&0&0
\\
\frac{(2H_r^2+9H_p-3DH_r)}{18}&\frac{H_r}{3}&-1&0&0
\\
0&0&0&1&F_p
\\
\frac{3F_{pp}F_{pppp}-5F_{ppp}^2}{18F_{pp}^2}&0&0&\frac{F_{ppp}}{3F_{pp}}&\frac{F_{ppp}F_p-3F_{pp}^2}{3F_{pp}}
\ema
\cdot
\bma
\om^1
\\
\om^2
\\
\om^3
\\
\om^4
\\
\om^5
\ema,
\]
and we invariantly reduce the $G$-structure to only $4$ parameters
$\rho$, $\phi$, $f^2$, $\bar{f}^2$\,\,---\,\,plus 
one extra parameter $u_1$\,\,---, the bar having
nothing to do with complex conjugation except some
analogy link with the CR computations 
in~{\cite{Merker-Pocchiola-2018}}:
\[
\bma
\theta^1\\\theta^2\\\theta^3\\\theta^4\\\theta^5\ema 
\,:=\,
\bma 
\rho^2&0&0&0&0
\\
f^2&\rho{\rm e}^\phi&0&0&0
\\
\frac{(f^2)^2}{2\rho^2}&\frac{f^2{\rm e}^\phi}{\rho}&{\rm e}^{2\phi}&0&0\\
\bar{f}{}^2&0&0&\rho{\rm e}^{-\phi}&0
\\
-\frac{(\bar{f}{}^2)^2}{2\rho^2}&0&0&
\frac{-\bar{f}{}^2{\rm e}^{-\phi}}{\rho}&{\rm e}^{-2\phi}\ema
\cdot
\bma
\om^1
\\
\om^2
\\
\om^3
\\
\om^4
\\
\om^5
\ema.
\]
After computational cleaning, we obtain our first result,
which happens to be the para-CR analog 
of~{\cite[Thm.~13.1]{Foo-Merker-2019}}.

\begin{theorem}
On the bundle $\mathcal{G}^9 = M^5 \times G^4$ with $M^5 \ni (x, y, z,
p, r)$ times $\R^4 \ni (\rho, \phi, f_2, \overline{f}_2)$, there exist
four $1$-forms $\Omega_1$, $\Omega_2$, $\Omega_3$, $\Omega_4$ with 
$\theta^1$, $\theta^2$, $\theta^3$, $\theta^4$, $\theta^5$, $\Omega_1$, $\Omega_2$,
$\Omega_3$, $\Omega_4$ linearly independent at every point which
satisfy the following para-CR invariant exterior differential system:
\begin{equation}
\label{d-theta-I-vert}
\begin{aligned}
\der\theta^1
&
\,=\,
-\,\theta^1\dz\Om_1+\theta^2\dz\theta^4,
\\
\der\theta^2
&
\,=\,
\theta^2\dz(\Om_2-\tfrac12\Om_1)
-
\theta^1\dz\Om_3+\theta^3\dz\theta^4,
\\
\der\theta^3
&
\,=\,
2\theta^3\dz\Om_2-\theta^2\dz\Om_3
+
\tfrac{{\rm e}^{3\phi}}{\rho^3}I^1\,\theta^1\dz\theta^4
+
\tfrac{{\rm e}^{-\phi}}{\rho}I^3\,\theta^2\dz\theta^3+
\\
&
\ \ \ \ \
\tfrac{1}{8\rho^3}\,
\Big(2\mathrm{e}^\phi\bar{f}{}^2I^3{}_{|5}
+
\rho(I^3{}_{|52}+2I^3{}_{|4})-4\mathrm{e}^{-\phi}f^2I^3
\Big)\,\theta^1\dz\theta^3,
\\
\der\theta^4
&
\,=\,
-\,\theta^2\dz\theta^5-\theta^4\dz(\tfrac12\Om_1+\Om_2)
-
\theta^1\dz\Om_4,
\\
\der\theta^5
&
\,=\,
-2\theta^5\dz\Om_2+\theta^4\dz\Om_4+\tfrac{{\rm e}^{-3\phi}}{\rho^3}I^2\,\theta^1\dz\theta^2-\tfrac{{\rm e}^{\phi}}{2\rho}I^3{}_{|5}\,\theta^4\dz\theta^5+
\\
&
\ \ \ \ \
\tfrac{1}{8\rho^3}\Big(2\mathrm{e}^\phi\bar{f}{}^2I^3{}_{|5}+\rho(I^3{}_{|52}+2I^3{}_{|4})-4\mathrm{e}^{-\phi}f^2I^3\Big)\,\theta^1\dz\theta^5,
\end{aligned}
\end{equation}
where $I^1$, $I^2$, $I^3$ are explicit relative differential
invariants on the base $M$:
\[
\aligned
I^1
&
\,:=\,
-\,
\tfrac{1}{54}\,
\big(
9D^2H_r-27DH_p-18DH_r H_r
+
18H_p H_r+4H_r^3+54H_z
\big),
\\
I^2
&
\,:=\,
\frac{40F_{ppp}^3-45 F_{pp}F_{ppp}F_{pppp}
+
9F_{pp}^2F_{ppppp}}{54\,F_{pp}^3},
\\
I^3
&
\,:=\,
\frac{2F_{ppp}+F_{pp}H_{rr}}{3\,F_{pp}},
\endaligned
\]
and where $(\cdot)\vert_i$ for $i = 1, \dots, 5$ denote directional
derivatives along the vector fields $X_i$ dual to $\theta^i$.
\end{theorem}

\hspace{-0.5cm}
\scalebox{0.85}{\begin{picture}(0,0)%
\includegraphics{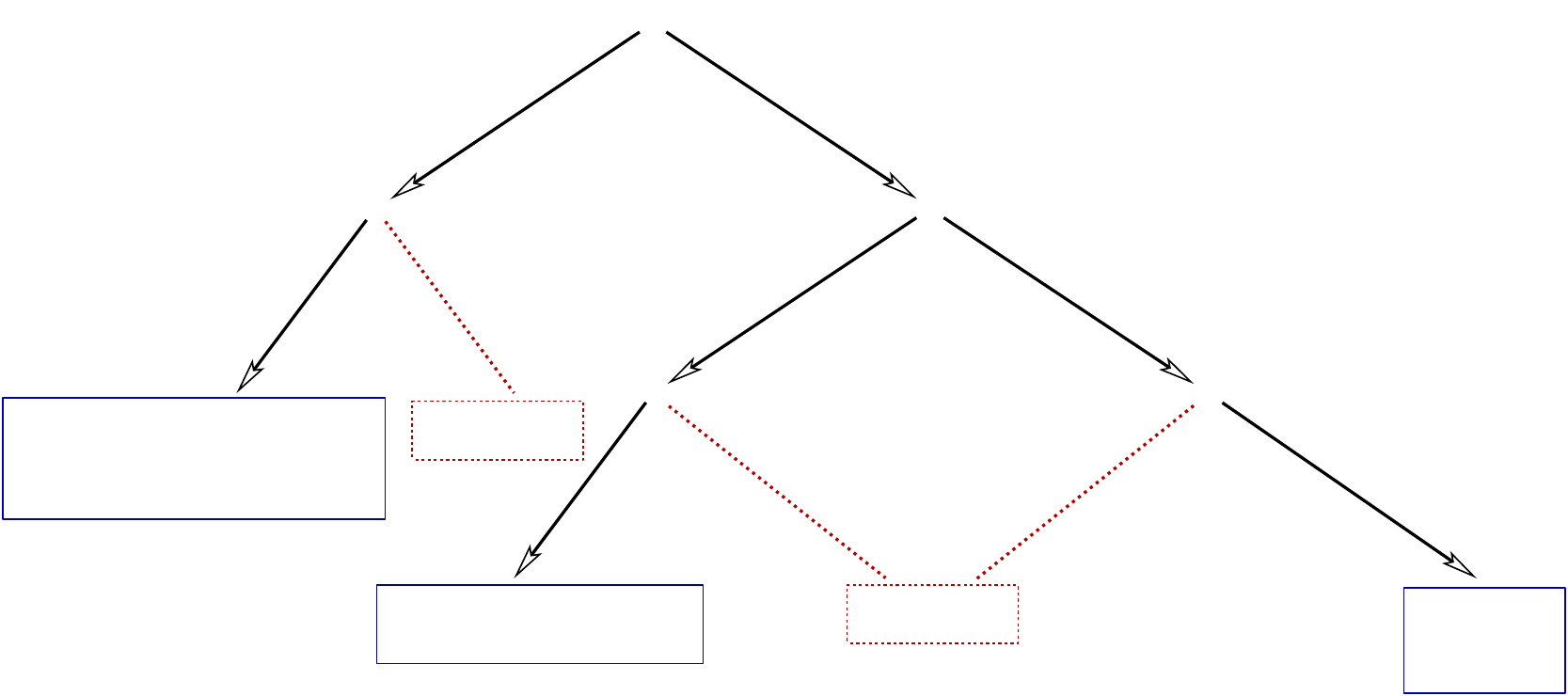}%
\end{picture}%
\setlength{\unitlength}{4144sp}%
\begingroup\makeatletter\ifx\SetFigFont\undefined%
\gdef\SetFigFont#1#2#3#4#5{%
  \reset@font\fontsize{#1}{#2pt}%
  \fontfamily{#3}\fontseries{#4}\fontshape{#5}%
  \selectfont}%
\fi\endgroup%
\begin{picture}(7618,3371)(878,-3329)
\put(3226,-434){\makebox(0,0)[lb]{\smash{{\SetFigFont{9}{10.8}{\familydefault}{\mddefault}{\updefault}{\color[rgb]{0,0,0}$0\neq$}%
}}}}
\put(4564,-1329){\makebox(0,0)[lb]{\smash{{\SetFigFont{9}{10.8}{\familydefault}{\mddefault}{\updefault}{\color[rgb]{0,0,0}$0\neq$}%
}}}}
\put(2137,-1374){\makebox(0,0)[lb]{\smash{{\SetFigFont{9}{10.8}{\familydefault}{\mddefault}{\updefault}{\color[rgb]{0,0,0}$0\neq$}%
}}}}
\put(4026,-81){\makebox(0,0)[lb]{\smash{{\SetFigFont{9}{10.8}{\familydefault}{\mddefault}{\updefault}{\color[rgb]{0,0,0}$I^3$}%
}}}}
\put(5369,-986){\makebox(0,0)[lb]{\smash{{\SetFigFont{9}{10.8}{\familydefault}{\mddefault}{\updefault}{\color[rgb]{0,0,0}$I^2$}%
}}}}
\put(6711,-1876){\makebox(0,0)[lb]{\smash{{\SetFigFont{9}{10.8}{\familydefault}{\mddefault}{\updefault}{\color[rgb]{0,0,0}$I^1$}%
}}}}
\put(3989,-1869){\makebox(0,0)[lb]{\smash{{\SetFigFont{9}{10.8}{\familydefault}{\mddefault}{\updefault}{\color[rgb]{0,0,0}$I^2_{\vert5}$}%
}}}}
\put(2612,-971){\makebox(0,0)[lb]{\smash{{\SetFigFont{9}{10.8}{\familydefault}{\mddefault}{\updefault}{\color[rgb]{0,0,0}$I^3_{\vert5}$}%
}}}}
\put(7750,-2940){\makebox(0,0)[lb]{\smash{{\SetFigFont{9}{10.8}{\familydefault}{\mddefault}{\updefault}{\color[rgb]{0,0,0}\blue{Flat model}}%
}}}}
\put(7749,-3264){\makebox(0,0)[lb]{\smash{{\SetFigFont{9}{10.8}{\familydefault}{\mddefault}{\updefault}{\color[rgb]{0,0,0}\blue{$z_{xxx}\!=\!0$}}%
}}}}
\put(7750,-3098){\makebox(0,0)[lb]{\smash{{\SetFigFont{9}{10.8}{\familydefault}{\mddefault}{\updefault}{\color[rgb]{0,0,0}\blue{$z_y\!=\!\frac{1}{4}z_{xx}^2$}}%
}}}}
\put(938,-2397){\makebox(0,0)[lb]{\smash{{\SetFigFont{9}{10.8}{\familydefault}{\mddefault}{\updefault}{\color[rgb]{0,0,0}\blue{$z_{xxx}\!=\!z_{xx}^3$}}%
}}}}
\put(2744,-2915){\makebox(0,0)[lb]{\smash{{\SetFigFont{9}{10.8}{\familydefault}{\mddefault}{\updefault}{\color[rgb]{0,0,0}\blue{$1$-parameter family}}%
}}}}
\put(935,-2012){\makebox(0,0)[lb]{\smash{{\SetFigFont{9}{10.8}{\familydefault}{\mddefault}{\updefault}{\color[rgb]{0,0,0}\blue{Single homogeneous model}}%
}}}}
\put(3399,-2379){\makebox(0,0)[lb]{\smash{{\SetFigFont{9}{10.8}{\familydefault}{\mddefault}{\updefault}{\color[rgb]{0,0,0}$0\neq$}%
}}}}
\put(2911,-2015){\makebox(0,0)[lb]{\smash{{\SetFigFont{8}{9.6}{\familydefault}{\mddefault}{\updefault}{\color[rgb]{0,0,0}\red{Differential}}%
}}}}
\put(935,-2199){\makebox(0,0)[lb]{\smash{{\SetFigFont{9}{10.8}{\familydefault}{\mddefault}{\updefault}{\color[rgb]{0,0,0}\blue{$z_y\!=\frac{1}{4}z_x^2$}}%
}}}}
\put(2901,-2158){\makebox(0,0)[lb]{\smash{{\SetFigFont{8}{9.6}{\familydefault}{\mddefault}{\updefault}{\color[rgb]{0,0,0}\red{ contradiction}}%
}}}}
\put(5025,-2908){\makebox(0,0)[lb]{\smash{{\SetFigFont{8}{9.6}{\familydefault}{\mddefault}{\updefault}{\color[rgb]{0,0,0}\red{Differential}}%
}}}}
\put(2742,-3107){\makebox(0,0)[lb]{\smash{{\SetFigFont{9}{10.8}{\familydefault}{\mddefault}{\updefault}{\color[rgb]{0,0,0}\blue{of homogeneous models}}%
}}}}
\put(5019,-3047){\makebox(0,0)[lb]{\smash{{\SetFigFont{8}{9.6}{\familydefault}{\mddefault}{\updefault}{\color[rgb]{0,0,0}\red{ contradiction}}%
}}}}
\put(3012,-1300){\makebox(0,0)[lb]{\smash{{\SetFigFont{9}{10.8}{\familydefault}{\mddefault}{\updefault}{\color[rgb]{0,0,0}\red{$\equiv 0$}}%
}}}}
\put(4525,-2184){\makebox(0,0)[lb]{\smash{{\SetFigFont{9}{10.8}{\familydefault}{\mddefault}{\updefault}{\color[rgb]{0,0,0}\red{$\equiv 0$}}%
}}}}
\put(5986,-2258){\makebox(0,0)[lb]{\smash{{\SetFigFont{9}{10.8}{\familydefault}{\mddefault}{\updefault}{\color[rgb]{0,0,0}\red{$0\neq$}}%
}}}}
\put(4617,-385){\makebox(0,0)[lb]{\smash{{\SetFigFont{9}{10.8}{\familydefault}{\mddefault}{\updefault}{\color[rgb]{0,0,0}$\equiv 0$}%
}}}}
\put(6041,-1332){\makebox(0,0)[lb]{\smash{{\SetFigFont{9}{10.8}{\familydefault}{\mddefault}{\updefault}{\color[rgb]{0,0,0}$\equiv 0$}%
}}}}
\put(7316,-2183){\makebox(0,0)[lb]{\smash{{\SetFigFont{9}{10.8}{\familydefault}{\mddefault}{\updefault}{\color[rgb]{0,0,0}$\equiv 0$}%
}}}}
\end{picture}%
}

Developing the technique of Cartan in 
{\em e.g.}~{\cite[Chap.~III]{Cartan-1932-I}},
we split the study in two branches: $I^3 \neq 0$ and
$I^3 \equiv 0$. When $I^3 \neq 0$, we show that
one can normalize $\rho$, $u_1$, $\bar{f}^2$.
Then in the obtained structure equations,
$I^3 \vert_5$ becomes a relative invariant.
We show that $I^3 \vert_5 \equiv 0$ conducts to
a differential contradiction.
When $I^3 \vert_5 \neq 0$, we can also normalize
$\phi$, $f^2$, hence obtaining an $\{e\}$-structure
on the base $M$, {\em cf.}~{\cite{Merker-Pocchiola-2018}}.
At first, certain $15$ scalar constant curvatures appear,
and by looking at differential consequences of
$d \circ d = 0$, they reduce to {\em only one pair of solutions},
with $\epsilon = \pm 1$, 
and we come to Maurer-Cartan type equations:
\[
\begin{aligned}
\der\theta^1=&\,\epsilon\Big(-6\theta^1\dz\theta^3+\tfrac12\theta^1\dz\theta^4-\tfrac32\theta^1\dz\theta^5\Big)+\theta^2\dz\theta^4,\\
\der\theta^2=&\,\epsilon\Big(-\tfrac{1}{16}\theta^1\dz\theta^2-2\theta^2\dz\theta^3+\tfrac12\theta^2\dz\theta^4-\theta^2\dz\theta^5\Big)-\theta^1\dz\theta^3+\\&\tfrac{1}{32}\theta^1\dz\theta^4-\tfrac18\theta^1\dz\theta^5+\theta^3\dz\theta^4,\\
\der\theta^3=&\,\epsilon\Big(-\tfrac{3}{16}\theta^1\dz\theta^3+\tfrac12\theta^3\dz\theta^4-\tfrac12\theta^3\dz\theta^5\Big)+\tfrac{1}{32}\theta^2\dz\theta^4-\tfrac18\theta^2\dz\theta^5,\\
\der\theta^4=&\,\epsilon\Big(-\tfrac18\theta^1\dz\theta^4+\tfrac14\theta^1\dz\theta^5+4\theta^3\dz\theta^4-\tfrac12\theta^4\dz\theta^5\Big)-\theta^2\dz\theta^5,\\
\der\theta^5=&\,\epsilon\Big(-\tfrac{1}{16}\theta^1\dz\theta^5+2\theta^3\dz\theta^5-\tfrac14\theta^4\dz\theta^5\Big).
\end{aligned}
\]

Next, in the branch $I^3 \equiv 0$, the 
equations~({\ref{d-theta-I-vert}}) become:
\[
\begin{aligned}
\der \theta^1=&-\theta^1\dz\Om_1+\theta^2\dz\theta^4,\\
\der \theta^2=&\theta^2\dz(\Om_2-\tfrac12\Om_1)-\theta^1\dz\Om_3+\theta^3\dz\theta^4,\\
\der \theta^3=&2\theta^3\dz\Om_2-\theta^2\dz\Om_3+\tfrac{{\rm e}^{3\phi}}{\rho^3}I^1\,\theta^1\dz\theta^4,\\
\der \theta^4=&-\theta^2\dz\theta^5-\theta^4\dz(\tfrac12\Om_1+\Om_2)-\theta^1\dz\Om_4,\\
\der \theta^5=&-2\theta^5\dz\Om_2+\theta^4\dz\Om_4+\tfrac{{\rm e}^{-3\phi}}{\rho^3}I^2\,\theta^1\dz\theta^2.
\end{aligned}
\]
Here, $I^1$ and $I^2$ are relative invariants.

In the sub-branch $I^2 \neq 0$, we first normalize
$\rho$, $u_1$, $\bar{f}^2$. 
Then $I^2 \vert_5$ becomes a relative invariant.
We show that $I^2 \vert_5 \equiv 0$ leads to a differential
contradiction. When $I^2 \vert_5 \neq 0$,
we can also normalize $\phi$, $f^2$,
hence obtaining an $\{e\}$-structure
on the base $M$, {\em cf.}~{\cite{Merker-Pocchiola-2018}}.
At first, certain $12$ scalar constant curvatures appear,
and by looking at differential consequences of
$d \circ d = 0$, they reduce to {\em one pair of
$1$-parameter solutions} 
and we come to Maurer-Cartan type equations,
parametrized by any $s \in \R$, again with 
$\epsilon = \pm 1$:
\[
\begin{aligned}
\der\theta^1=&-\epsilon\Big(\theta^1\dz\theta^3+\theta^1\dz\theta^5\Big)+\theta^2\dz\theta^4,\\
\der\theta^2=&\,\epsilon\Big(s\theta^1\dz\theta^2-\theta^2\dz\theta^5\Big)-s\theta^1\dz\theta^4+\theta^3\dz\theta^4,\\
\der\theta^3=&\,\epsilon\Big(\theta^1\dz\theta^4-\theta^3\dz\theta^5\Big)-\theta^1\dz\theta^2-s\theta^2\dz\theta^4,\\
\der\theta^4=&\,\epsilon\Big(-s\theta^1\dz\theta^4+\theta^3\dz\theta^4\Big)+s\theta^1\dz\theta^2-\theta^2\dz\theta^5,\\
\der\theta^5=&\,\epsilon\Big(-\theta^1\dz\theta^4+\theta^3\dz\theta^5\Big)+\theta^1\dz\theta^2+s\theta^2\dz\theta^4.
\end{aligned}
\]

Lastly, when $I^2 \equiv 0$, we show that $I^1 \equiv 0$ too
necessarily, and we show that the structure
equations are those of the model $z_y = \frac{1}{4}\,
(z_{xx})^2$ $\,\,\&\,\,$ $z_{xxx} = 0$.
The diagram above summarizes these explanations.

By general features of Cartan's method, all obtained
para-CR structures are pairwise not equivalent. 

To conclude, by setting up the PDEs
associated to para-CR
submanifolds of solutions inspired from Fels-Kaup's 
list~{\cite{Fels-Kaup-2008}},
we realize all these homogeneous models as stated in our main

\begin{theorem}
\label{main-theorem}
Homogeneous models for $2$-nondegenerate PDE five variables 
para-CR structures are classified by
the following list of mutually inequivalent models:

\smallskip\noindent
{\bf (i)}\,\, 
$z_y = \tfrac14 (z_x)^2\quad \&\quad z_{xxx}=0$;

\smallskip\noindent
{\bf (ii)}\,\, 
$z_y = \tfrac14 (z_x)^2\quad \& \quad z_{xxx}=(z_{xx})^3$;

\smallskip\noindent
{\bf (iiia)}\,\, 
$z_y = \tfrac14 (z_x)^b\,\, \& \,\,z_{xxx} = 
(2-b)\frac{(z_{xx})^2}{z_x}$ with $z_x > 0$ 
for any real $b \in [1, 2)$;

\smallskip\noindent
{\bf (iiib)}\,\, 
$z_y = f(z_x)\quad \& \quad z_{xxx} = h(z_x)\big(z_{xx}\big)^2$, 
where the function $f$ is determined by the implicit equation:
\[
\big(z_x^2+f(z_x)^2\big)
\mathrm{exp}
\Big(
2b\,\mathrm{arctan}\tfrac{bz_x-f(z_x)}{z_x+bf(z_x)}
\Big)
\,=\,
1+b^2
\] 
and where:
\[
h(z_x)
\,:=\,
\frac{(b^2-3)z_x-4bf(z_x)}{\big(f(z_x)-bz_x\big)^2},
\]
for any real $b > 0$.
\end{theorem}

Our explorations can certainly
be generalized to higher dimensions, 
{\em cf.}~{\cite{Porter-2015, Porter-Zelenko-2017}} in a CR context.

The body of the paper is devoted to provide a streamlined
exposition of our Cartan-type techniques. In 
Section~{\ref{Lie-symmetry-algebras}}, the reader will find
the Lie algebras of point symmetries of these models
{\bf (i)}, {\bf (ii)}, {\bf (iiia)}, {\bf (iiib)}.

To end up this introduction, recall that the complete classification
of $A_3(\R)$-homogeneous surfaces $S^2 \subset \R^3$ was terminated by
Doubrov-Komrakov-Rabinovich
in~{\cite{Doubrov-Komrakov-Rabinovich-1996}},
and re-done by Eastwood-Ezhov
in~{\cite{Eastwood-Ezhov-1999}}, who used the power series method.
The full classification
includes that of $A_3(\R)$-homogeneous {\em parabolic}
surfaces $\{ u = F (x,y)\}$ in $\R^3$ having Hessian matrix
$\big( \begin{smallmatrix} F_{xx} & F_{xy} \\ F_{yx} &
F_{yy} \end{smallmatrix} \big)$ everywhere of rank $1$.
The classification lists contained in~{\cite{Eastwood-Ezhov-1999,
Doubrov-Komrakov-Rabinovich-1996}} are
in accordance with our main Theorem~{\ref{main-theorem}}. 
The last Section~{\ref{link-homogeneous-parabolic-surfaces}}
compares these classifications and shows that, 
surprisingly, Cartan's reduction 
gathers a scattered number of
models into just one family.

\medskip\noindent{\bf Acknowledgments.}
This work would not have been realized without the generous support
of the Polish National Science Center. Hoping to benefit from
renewed excellent working conditions, 
the authors would also like to thank the 
{\sl Center for Theoretical Physics of the Polish Academy of Sciences}
in Warsaw and the
{\sl Institut de Math\'ematique d'Orsay} in Paris.

\section{Preliminaries}\label{prelim}

\subsection{Five dimensional para-CR manifolds 
with Levi form degenerate in one direction}\label{system}
A 5-dimensional CR manifold whose Levi form is degenerate in precisely
one direction, and which is \emph{not} locally CR-isomorphic to a
product of a 3-dimensional CR manifold times $\bbC$, is called
\emph{$2$-nondegenerate} at a generic point. It is well known that the
flat model for 5-dimensional 2-nondegenerate CR manifolds, is a `tube
over the future light cone' \cite{Freeman-1977}, and as such can be
embedded in $\bbC^3$ with coordinates $(x,y,z)$ as:
\be
(x+\bar{x})^2+(y-\bar{y})(z-\bar{z})=0.\label{tube}
\ee

This CR manifold is \emph{flat} in the sense that it has maximal group
of local symmetries among all 5-dimensional 2-nondegenerate CR
manifolds. This symmetry group is isomorphic to 
$\sog(3,2)$. In other words all
5-dimensional 2-nondegenerate CR manifolds are described in terms of a
Cartan reduction to an $\{e\}$-structure, with flat model having
Maurer-Cartan equations of the Lie group $\sog(3,2)$, and a CR
manifold is locally equivalent to the \emph{tube over the future light
cone} if and only if the curvature of this connection identically
vanish ({\cite{
Medori-Spiro-2014, Merker-Pocchiola-2018}}).

In this paper we will study a \emph{para}-CR version of 5-dimensional
2-nondegenerate CR manifolds. As explained in details in 
\cite{Hill-Nurowski-2010,
  Merker-2008} a geometry of para-CR manifolds is closely related to
the geometry of certain systems of PDEs. To see this consider the tube
over the future light cone (\ref{tube}) and think about variables
$(x,y,z,\bar{x},\bar{y},\bar{z})$ as beeing \emph{real},
i.e. $(x,y,z,\bar{x},\bar{y},\bar{z})\in\bbR^6=\bbR^3\oplus\overline{\bbR}{}^3$,
where we have put a bar over the second $\bbR^3$ in the summand, to
emphasize the difference between the \emph{real} variables $(x,y,z)$
and the \emph{real} variables $(\bar{x},\bar{y},\bar{z})$.

Treating $(x,y,z,\bar{x},\bar{y},\bar{z})$ in (\ref{tube}) as real, we solve this equation for $z$ obtaining
\be 
z
=
-\frac{(x+\bar{x})^2}{y-\bar{y}}+\bar{z}.\label{flatsol}
\ee
And now we interpret this expression as a defining formula for a
3-parameter family of \emph{functions}
$z=z(x,y;\bar{x},\bar{y},\bar{z})$ on the plane $(x,y)$, with
$(\bar{x},\bar{y},\bar{z})$ enumerating the members of the family. We
calculate the derivatives $z_x$, $z_y$ and $z_{xxx}$ and observe that
regardless of $(\bar{x},\bar{y},\bar{z})$ we have
\be 
z_y
=
\tfrac14 (z_x)^2\quad\quad \& \quad\quad z_{xxx}=0.
\label{flateq}
\ee
Conversely, a system of PDEs on the plane (\ref{flateq}) for the
unknown $z=z(x,y)$ has (\ref{flatsol}) as its most general solution.

Para-CR structures associated with the system of PDEs defined in the
title of this article and, in particular, the para-CR structure
associated with the system \eqref{flateq}, according to Definition~2.3
from \cite{Hill-Nurowski-2010}, is of type $(1,2,2)$ i.e. is defined
in terms of an equivalence class $[(\om^1,\om^2,\om^3,\om^4,\om^5)]$
of 1-forms on a 5-dimensional manifold $M$ such that:
\begin{itemize}
\item $\om^1\dz\om^2\dz\om^3\dz\om^4\dz\om^5\neq 0$ 
at each point of $M$,
\item two choices of 1-forms
  $(\om^1,\om^2,\om^3,\om^4,\om^5)$
  and $(\bar{\om}{}^1,\bar{\om}{}^2,\bar{\om}{}^3,\bar{\om}{}^4,\bar{\om}{}^5)$ are equivalent iff there exist real functions $f$, $f^i$, 
${f^i}_j$, with $i,j=2,3,4,5$, such that:
\be
\bma \bar{\om}{}^1\\\bar{\om}{}^1\\\bar{\om}{}^1\\\bar{\om}{}^1\\\bar{\om}{}^1\ema=\bma f&0&0&0&0\\ f^2&f^2{}_2&f^2{}_3&0&0\\ f^3&f^3{}_2&f^3{}_3&0&0\\ f^4&0&0&f^4{}_4&f^4{}_5\\f^5&0&0&f^5{}_4&f^5{}_5\ema\bma \om^1\\\om^2\\\om^3\\\om^4\\\om^5\ema\label{rot}
\ee
with of course 
$f(f^2{}_2f^3{}_3-f^2{}_3f^3{}_2)
(f^4{}_4f^5{}_5-f^4{}_5f^5{}_4)\neq 0$,

\item in addition, any (hence all) representative(s) $(\bar{\om}{}^1,\bar{\om}{}^2,\bar{\om}{}^3,\bar{\om}{}^4,\bar{\om}{}^5)$ of the equivalence class $[(\om^1,\om^2,\om^3,\om^4,\om^5)]$ must satisfy 
\emph{integrability conditions}
  \be\begin{aligned}
&\der\bar{\om}{}^1\dz\bar{\om}{}^1\dz\bar{\om}{}^2\dz\bar{\om}{}^3=0,\\
  &\der\bar{\om}{}^2\dz\bar{\om}{}^1\dz\bar{\om}{}^2\dz\bar{\om}{}^3=0,\\
  &\der\bar{\om}{}^3\dz\bar{\om}{}^1\dz\bar{\om}{}^2\dz\bar{\om}{}^3=0,\\
  &\der\bar{\om}{}^1\dz\bar{\om}{}^1\dz\bar{\om}{}^4\dz\bar{\om}{}^5=0,\\
  &\der\bar{\om}{}^4\dz\bar{\om}{}^1\dz\bar{\om}{}^4\dz\bar{\om}{}^5=0,\\
  &\der\bar{\om}{}^5\dz\bar{\om}{}^1\dz\bar{\om}{}^4\dz\bar{\om}{}^5=0,\label{integ}
  \end{aligned}
  \ee
  hence defining two integrable rank 2 distributions ${\mathcal D}_1=(\omega^1,\omega^2,\omega^3)^\perp$ and ${\mathcal D}_2=(\omega^1,\omega^4,\omega^5)^\perp$ on $M$.
\end{itemize}

A para-CR structure has the \emph{Levi form degenerate in precisely one direction} if and only if in the class of forms \eqref{rot} there exists a representative $(\bar{\om}{}^1,\bar{\om}{}^2,\bar{\om}{}^3,\bar{\om}{}^4,\bar{\om}{}^5)$ such that
$$\der\bar{\om}{}^1\dz\bar{\om}{}^1=\bar{\om}{}^2\dz\bar{\om}{}^4\dz\bar{\om}{}^1.$$
In the case when a para-CR structure has a Levi form degenerate in precisely one direction, it is \emph{2-nondegenerate} if and only if in the class of forms \eqref{rot} there exists a representative $(\bar{\om}{}^1,\bar{\om}{}^2,\bar{\om}{}^3,\bar{\om}{}^4,\bar{\om}{}^5)$ such that
$$\der\bar{\om}{}^1\dz\bar{\om}{}^1=\bar{\om}{}^2\dz\bar{\om}{}^4\dz\bar{\om}{}^1\quad\&\quad \der\bar{\om}{}^4\dz\bar{\om}{}^1\dz\bar{\om}{}^4\neq 0.$$

Given a PDE system (\ref{flateq}) one can consider a 5-dimensional manifold $M$ of second jets for the function $z$, parameterized by $(x,y,z,z_x,z_{xx})$, and define 1-forms
\be\begin{aligned}
  \bar{\om}{}^1&=\der z- z_x\der x-\tfrac14 (z_x)^2\der y\\
  \bar{\om}{}^2&=\der z_x-z_{xx}\der x-\tfrac12z_xz_{xx}\der y\\
  \bar{\om}{}^3&=\der z_{xx}-\tfrac12 (z_{xx})^2\der y\\
  \bar{\om}{}^4&=\der x\\
  \bar{\om}{}^5&=\der y.
\end{aligned}\label{flatforms}\ee
One can easilly verify that they satisfy the integrability conditions \eqref{integ}. Thus, they define a $(1,2,2)$ type para-CR structure on $M$ by considering all five-tuples of 1-forms $(\om^1,\om^2,\om^3,\om^4,\om^5)$ given by
\be\bma
\om^1\\\om^2\\\om^3\\\om^4\\\om^5\ema =\bma f&0&0&0&0\\ f^2&f^2{}_2&f^2{}_3&0&0\\ f^3&f^3{}_2&f^3{}_3&0&0\\ f^4&0&0&f^4{}_4&f^4{}_5\\f^5&0&0&f^5{}_4&f^5{}_5\ema\bma\bar{\om}{}^1\\\bar{\om}{}^2\\\bar{\om}{}^3\\\bar{\om}{}^4\\\bar{\om}{}^5\ema,\label{eqa}\ee
with arbitrary functions $f,f^i {f^i}_j$ on $M$ such that
\be
f(f^2{}_2f^3{}_3-f^2{}_3f^3{}_2)(f^4{}_4f^5{}_5-f^4{}_5f^5{}_4)\neq 0.\label{eqb}\ee As explained in \cite{Hill-Nurowski-2010} this para-CR manifold describes the same differential geometry as the system of PDEs \eqref{flateq} considered modulo \emph{point} transformations of variables. 
\subsection{The flat model and its EDS}\label{flatm}
Let $(\bar{\om}{}^1,\bar{\om}{}^2,\bar{\om}{}^3,\bar{\om}{}^4,\bar{\om}{}^5)$ be the forms \eqref{flatforms} defining the para-CR structure corresponding to the PDE system \eqref{flateq}. We use an equivalent representative of these forms given by 
$$\bma
\om^1\\\om^2\\\om^3\\\om^4\\\om^5\ema =\bma -1&0&0&0&0\\ 0&1&0&0&0\\ 0&0&-1&0&0\\ 0&0&0&1&\tfrac12 z_x\\ 0&0&0&0&-\tfrac12\ema\bma\bar{\om}{}^1\\\bar{\om}{}^2\\\bar{\om}{}^3\\\bar{\om}{}^4\\\bar{\om}{}^5\ema,$$ i.e.
\be
\begin{aligned}
  \om^1&=-\der z+ z_x\der x+\tfrac14 (z_x)^2\der y\\
  \om^2&=\der z_x-z_{xx}\der x-\tfrac12z_xz_{xx}\der y\\
  \om^3&=-\der z_{xx}+\tfrac12 (z_{xx})^2\der y\\
  \om^4&=\der x+\tfrac12z_x\der y\\
  \om^5&=-\tfrac12\der y.
  \end{aligned}\label{flfo}
\ee
They satisfy the system:
$$\begin{aligned}
  \der\om^1&=\om^2\dz\om^4\\
  \der\om^2&=z_{xx}\om^2\dz\om^5+\om^3\dz\om^4\\
  \der\om^3&=2z_{xx}\om^3\dz\om^5\\
  \der\om^4&=-\om^2\dz\om^5-z_{xx}\om^4\dz\om^5\\
  \der\om^5&=0.
  \end{aligned}
$$
These equations show, in particular, that the para-CR structure defined by the PDE system \eqref{flateq} has the \emph{Levi form degenerate in precisely one direction} and that it is \emph{2-nondegenerate}.

For reasons which will be clear in the proof of Theorem \ref{th1.1} it is convenient to define the following auxiliary 1-forms:
\be
\varpi_1=0,\,\,\varpi_2=z_{xx}\om^5,\,\,\varpi_3=0,\,\,\varpi_4=0,\,\,\varpi_5=0.\label{ofo}\ee
Although majority of these forms are vanishing, they will not vanish in the case of a general system of PDEs defined in the title of this article.


With these auxiliary forms the system of ten 1-forms
$(\om^1,\om^2,\om^3,\om^4,\om^5,\varpi_1,\varpi_2,$ $\varpi_3,\varpi_4,$ $\varpi_5)$ on $M$ satisfies an EDS:
\be\begin{aligned}
\der \om^1&=\om^2\dz\om^4-\om^1\dz\varpi_1\\
\der \om^2&=\om^3\dz\om^4+\om^2\dz(\varpi_2-\tfrac12\varpi_1)-\om^1\dz\varpi_3\\
\der \om^3&=2\om^3\dz\varpi_2-\om^2\dz\varpi_3\\
\der \om^4&=-\om^2\dz\om^5-\om^4\dz(\tfrac12\varpi_1+\varpi_2)-\om^1\dz\varpi_4\\
\der \om^5&=-2\om^5\dz\varpi_2+\om^4\dz\varpi_4\\
\der \varpi_1&=-\om^4\dz\varpi_3+\om^2\dz\varpi_4-\om^1\dz\varpi_5\\
\der \varpi_2&=-\om^3\dz\om^5-\tfrac12\om^4\dz\varpi_3-\tfrac12\om^2\dz\varpi_4\\
\der \varpi_3&=-(\tfrac12\varpi_1+\varpi_2)\dz\varpi_3+\om^3\dz\varpi_4-\tfrac12\om^2\dz\varpi_5\\
\der \varpi_4&=(\varpi_2-\tfrac12\varpi_1)\dz\varpi_4+\om^5\dz\varpi_3-\tfrac12\om^4\dz\varpi_5\\
\der \varpi_5&=-\varpi_1\dz\varpi_5+2\varpi_3\dz\varpi_4.
\end{aligned}\label{sys0}\ee
Now we consider the most general forms defining the para-CR structure corresponding to the PDE system \eqref{flateq}. These are
\be
\bma
\theta^1\\\theta^2\\\theta^3\\\theta^4\\\theta^5\ema =\bma f&0&0&0&0\\ f^2&f^2{}_2&f^2{}_3&0&0\\ f^3&f^3{}_2&f^3{}_3&0&0\\ f^4&0&0&f^4{}_4&f^4{}_5\\f^5&0&0&f^5{}_4&f^5{}_5\ema\bma\om^1\\\om^2\\\om^3\\\om^4\\\om^5\ema,\label{mog}
\ee
with $(\om^1,\om^2,\om^3,\om^4,\om^5)$ as in \eqref{flfo}. These forms live on a $(5+13)$-dimensional bundle $M\times G_0\to M$, with a group $G_0=\big(\glg(2,\bbR)\times\glg(2,\bbR)\big)\rtimes\bbR^5$ consisting of all matrices of the form:
$$S=\bma f&0&0&0&0\\ f^2&f^2{}_2&f^2{}_3&0&0\\ f^3&f^3{}_2&f^3{}_3&0&0\\ f^4&0&0&f^4{}_4&f^4{}_5\\f^5&0&0&f^5{}_4&f^5{}_5\ema\,\,\rm{such\,\,that}\,\, \det(S)\neq 0.$$

\begin{theorem}\label{th1.1}
  The para-CR structure $[(\om^1,\om^2,\om^3,\om^4,\om^5)]$ defined on $M$ by a representative $(\om^1,\om^2,\om^3,\om^4,\om^5)$ as in \eqref{flfo} locally uniquely defines a 10-dimensional principal bundle ${\mathcal G}=M\times G\to M$, with a 5-dimensional Lie group $G$ consisting of all matrices of the form
  \be
  U=\bma \frac{{\rm e}^{-\phi}}{r}&-\frac{{\rm e}^{-\phi}(s\bar{s}+r^4 u)}{2r^3}&-\frac{\bar{s}}{r^2}&0&-\frac{{\rm e}^\phi\bar{s}{}^2}{2r^3}\\&&&&\\ 0&r{\rm e}^{-\phi}&0&0&0\\&&&&\\0&\frac{s{\rm e}^{-\phi}}{r}&1&0&\frac{\bar{s}{\rm e}^{\phi}}{r}\\&&&&\\0&\frac{s^2{\rm e}^{-\phi}}{2r^3}&\frac{s}{r^2}&\frac{{\rm e}^{\phi}}{r}&\frac{{\rm e}^{\phi}(s\bar{s}-r^4 u)}{2r^3}\\&&&&\\ 0&0&0&0&r{\rm e}^\phi\ema,\label{U}
  \ee
and a rigid coframe $(\theta^1,\theta^2,\theta^3,\theta^4,\theta^5,\Om_1,\Om_2,\Om_3,\Om_4,\Om_5)$ on $\mathcal G$ satisfying:
  \be\begin{aligned}
\der \theta^1&=\theta^2\dz\theta^4-\theta^1\dz\Om_1\\
\der \theta^2&=\theta^3\dz\theta^4+\theta^2\dz(\Om_2-\tfrac12\Om_1)-\theta^1\dz\Om_3\\
\der \theta^3&=2\theta^3\dz\Om_2-\theta^2\dz\Om_3\\
\der \theta^4&=-\theta^2\dz\theta^5-\theta^4\dz(\tfrac12\Om_1+\Om_2)-\theta^1\dz\Om_4\\
\der \theta^5&=-2\theta^5\dz\Om_2+\theta^4\dz\Om_4\\
\der \Om_1&=-\theta^4\dz\Om_3+\theta^2\dz\Om_4-\theta^1\dz\Om_5\\
\der \Om_2&=-\theta^3\dz\theta^5-\tfrac12\theta^4\dz\Om_3-\tfrac12\theta^2\dz\Om_4\\
\der \Om_3&=-(\tfrac12\Om_1+\Om_2)\dz\Om_3+\theta^3\dz\Om_4-\tfrac12\theta^2\dz\Om_5\\
\der \Om_4&=(\Om_2-\tfrac12\Om_1)\dz\Om_4+\theta^5\dz\Om_3-\tfrac12\theta^4\dz\Om_5\\
\der \Om_5&=-\Om_1\dz\Om_5+2\Om_3\dz\Om_4.
\end{aligned}\label{edsf}\ee
\end{theorem}
\begin{proof}
The proof of this theorem follows from the observation that the forms  $(\om^1,\om^2,$ $\om^3,$ $\om^4,\om^5,\varpi_1,\varpi_2,\varpi_3,\varpi_4,\varpi_5)$ satisfying the EDS \eqref{sys0} constitute a pullback
$$B=
\bma
\tfrac12\varpi_1-\varpi_2&-\tfrac12\varpi_5&\varpi_4&\om^5&0\\
\om^1&-\tfrac12\varpi_1-\varpi_2&\om^4&0&\om^5\\
\om^2&-\varpi_3&0&\om^4&-\varpi_4\\
\om^3&0&-\varpi_3&\tfrac12\varpi_1+\varpi_2&-\tfrac12\varpi_5\\
0&\om^3&-\om^2&\om^1&-\tfrac12\varpi_1+\varpi_2
\ema
$$
to $M$, by an identity section $\sigma:M\to{\mathcal G}$, of a \emph{flat} $\soa(3,2)$-valued Cartan connection
\be
\omega =
\bma
\tfrac12\Om_1-\Om_2&-\tfrac12\Om_5&\Om_4&\theta^5&0\\
\theta^1&-\tfrac12\Om_1-\Om_2&\theta^4&0&\theta^5\\
\theta^2&-\Om_3&0&\theta^4&-\Om_4\\
\theta^3&0&-\Om_3&\tfrac12\Om_1+\Om_2&-\tfrac12\Om_5\\
0&\theta^3&-\theta^2&\theta^1&-\tfrac12\Om_1+\Om_2
\ema
\label{carcon}\ee
on the bundle ${\mathcal G}\stackrel{\pi}{\to}M$. The relation between the pullback $B$ and the Cartan connection $\omega$ is given by
\be
\omega=U\cdot\pi^*(B)\cdot U^{-1}-\der{U}\cdot U^{-1},\label{UB}\ee
with $U$ given by \eqref{U}. On the identity section we have $U=\id$, and $\omega=B$.
Relation \eqref{UB}, when written component by component, gives
$$\bma
\theta^1\\\theta^2\\\theta^3\\\theta^4\\\theta^5\ema =\bma r^2&0&0&0&0\\ s&r{\rm e}^\phi&0&0&0\\\frac{s^2}{2r^2}&\frac{s{\rm e}^\phi}{r}&{\rm e}^{2\phi}&0&0\\ \bar{s}&0&0&r{\rm e}^{-\phi}&0\\ -\frac{\bar{s}{}^2}{2r^2}&0&0&-\frac{\bar{s}{\rm e}^{-\phi}}{r}&{\rm e}^{-2\phi}\ema\cdot\bma\om^1\\\om^2\\\om^3\\\om^4\\\om^5\ema,$$
and
$$\begin{aligned}
  \Om_1&=\frac{2\der r}{r}-u r^2\om^1-\frac{\bar{s}{\rm e}^\phi}{r}\om^2+\frac{s{\rm e}^{-\phi}}{r}\om^4,\\
  \Om_2&=-\der \phi+\frac{s\bar{s}}{2r^2}\om^1+\frac{\bar{s}{\rm e}^\phi}{2r}\om^2+\frac{s{\rm e}^{-\phi}}{2r}\om^4+z_{xx}\om^5,\\
  \Om_3&=\frac{\der s}{r^2}-\frac{s}{r^2}(\der\phi+\frac{\der r}{r})-\tfrac12s u\om^1-\frac{{\rm e}^{\phi}(s\bar{s}+r^4 u)}{2r^3}\om^2-\frac{\bar{s}{\rm e}^{2\phi}}{r^2}\om^3+\frac{{\rm e}^{-\phi}s^2}{2r^3}\om^4+\frac{s z_{xx}}{r^2}\om^5,\\
  \Om_4&=\frac{\der \bar{s}}{r^2}+\frac{\bar{s}}{r^2}(\der\phi-\frac{\der r}{r})-\tfrac12\bar{s} u\om^1-\frac{\bar{s}^2{\rm e}^{\phi}}{2r^3}\om^2+\frac{{\rm e}^{-\phi}(s\bar{s}-r^4 u)}{2r^3}\om^4-\frac{{\rm e}^{-2\phi}(s+{\rm e}^{2\phi}\bar{s}z_{xx})}{r^2}\om^5,\\
  \Om_5&=-\der u-\frac{2u\der r}{r}+\frac{2s\bar{s}\der \phi}{r^4}+\frac{s\der\bar{s}}{r^4}-\frac{\bar{s}\der s}{r^4}+\\
&\quad \tfrac12r^2 u^2\om^1+\frac{\bar{s}{\rm e}^{\phi}u}{r}\om^2+\frac{{\rm e}^{2\phi}\bar{s}{}^2}{r^4}\om^3-\frac{{\rm e}^{-\phi}su}{r}\om^4-\frac{s{\rm e}^{-2\phi}(s+2{\rm e}^{2\phi}\bar{s}z_{xx})}{r^4}\om^5.
  \end{aligned}
$$
\noindent
In these expressions the forms $(\om^1,\om^2,\om^3,\om^4,\om^5)$ are as in \eqref{flfo}. Check, in particular, that on the identity section given by $r=1$, $\phi=0$, $f=0$, $\bar{f}=0$, $u=0$, the forms $(\theta^1,\theta^2,\theta^3,\theta^4,\theta^5,\Om_1,\Om_2,\Om_3,\Om_4,\Om_5)$ become respectively $(\om^1,\om^2,\om^3,\om^4,\om^5,\varpi_1=0,\varpi_2=z_{xx}\om^5,\varpi_3=0,\varpi_4=0,\varpi_5=0$), which explains why we introduced the forms $\omega_i$ in \eqref{ofo}. 

The fact that the above coframe $(\theta^1,\theta^2,\theta^3,\theta^4,\theta^5,\Om_1,\Om_2,\Om_3,\Om_4,\Om_5)$ on $\mathcal G$ satisfies the EDS \eqref{edsf} is equivalent to the following equality
$$\der\omega+\omega\dz\omega=0, $$
satisfied by the Cartan connection $\omega$. This can be checked by a direct calculation using the explicit expressions for $(\theta^1,\theta^2,\theta^3,\theta^4,\theta^5,\Om_1,\Om_2,\Om_3,\Om_4,\Om_5)$ given above.

Thus, the Cartan connection $\omega$ given by \eqref{carcon} is flat, and the $\soa(3,2)$-valued 1-form $\omega$ can be interpreted as a Maurer-Cartan form on the group $\sog(3,2)$. The Cartan bundle ${\mathcal G}\to M$ is then identified as a realization of the homogeneous model $G\to \sog(3,2)\to M=\sog(3,2)/G$, which has a natural para-CR structure related to forms $(\om^1,\om^2,\om^3,\om^4,\om^5)$ being in the same equivalence class as the respective descendent forms $(\theta^1,\theta^2,\theta^3,\theta^4,\theta^5)$. Obviously this structure has $\sog(3,2)$ as its group of symmetries.\end{proof}

\section{Nonflat case; four basic invariants}

Now we generalize the flat example of Subsections 
\ref{system}-\ref{flatm}  
to systems of PDEs on the plane of
the form 
\be 
z_{y}=F(x,y,z,z_x,z_{xx})\quad\quad\&\quad\quad
z_{xxx}=H(x,y,z,z_x,z_{xx}).\label{sysf} \ee
We introduce the standard notation $$p=z_x,\quad q=z_y,\quad r=z_{xx},$$
i.e. we have
$$z_y=F(x,y,z,p,r),\quad\quad\&\quad\quad z_{xxx}=H(x,y,z,p,r).$$
We note that 
for this system of equations to be equivalent to a 2-nondegenerate para-CR manifold we have to assume $F_r=0$ and $F_{pp}\neq 0$.
In addition, this system is of finite type, or, what is the
same, its general solution can be written as $z=z(x,y;\bar{x},\bar{y},\bar{z})$,
if and only if $D^3F=\Delta H$, 
with 
\be 
D=\partial_x+p\partial_z+r\partial_p+H\partial_r\quad\quad\&\quad\quad
\Delta=\partial_y+F\partial_z+DF\partial_p+D^2F\partial_r.\label{pndha}
\ee
From now on, we consider only systems \eqref{sysf} satisfying
\be
F_r=0,\quad\quad\& \quad\quad F_{pp}\neq 0,\quad\quad\&\quad\quad D^3F=\Delta H.
\label{sysfc}\ee

We now define a 2-nondegenerate para-CR structure on a 5-manifold $M$ associated with the equations \eqref{sysf}, \eqref{sysfc} by introducing an equivalence class of 1-forms as in \eqref{eqa}-\eqref{eqb}, but this time with forms  $(\om^1,\om^2,\om^3,\om^4,\om^5)$ given by:
\be
\begin{aligned}
  \om^1&=\der z-p\der x-F\der y\\
  \om^2&=\der p-r\der x-DF\der y\\
  \om^3&=\der r-H\der x-D^2F\der y\\
  \om^4&=\der x\\
  \om^5&=\der y.  
  \end{aligned}
\label{iniom}\ee
These forms live on a manifold $M$  parameterized by $(x,y,z,p,r)$, which is the 5-dimensional manifold of second jets for functions $z=z(x,y)$. The differentials of the initial forms are as follows:
\be
\begin{aligned}
  \der\om^1=&-F_z\om^1\dz\om^5-\om^2\dz\om^4-F_p\om^2\dz\om^5,\\
  \der\om^2=&-DF_z\om^1\dz\om^5-(DF_p+F_z)\om^2\dz\om^5-\om^3\dz\om^4-F_p\om^3\dz\om^5,\\
  \der\om^3=&-H_z\om^1\dz\om^4-(DDF_z+F_pH_z)\om^1\dz\om^5-H_p\om^2\dz\om^4+\\&\tfrac13(DF_p H_r-3DF_z-\Delta H_r+DH_r F_p-3F_pH_p)\om^2\dz\om^5-H_r\om^3\dz\om^4-\\&(2DF_p+F_z+F_pH_r)\om^3\dz\om^5,\\
  \der\om^4=&\,0,\\
  \der\om^5=&\,0.
  \end{aligned}\label{iniomi}
\ee
Here we introduce abbreviations such as $\Delta H_r$, or $DDF_z$, and abbreviations analogous to them. They mean:
$$\Delta H_r=\Delta(\partial_rH)\quad\mathrm{and}\quad DDF_z=D(D(\partial_zF)).$$
\begin{definition}
A 5-dimensional para-CR structure related to the point equivalence class of PDEs \eqref{sysf} satisfying \eqref{sysfc} via the representatives \eqref{iniom} will be called \emph{PDE five variables para-CR structure}.   
  \end{definition}

Now, till the end of this Section, will adopt the convention that if $f$ is a differentiable function on $M$, then its \emph{coframe derivatives} will be denoted by a subscript running from 1 to 5:
$$\der f=: f_1\om^1+f_2\om^2+f_3\om^3+f_4\om^4+f_5\om^5,$$
i.e.
\be f_\mu=\frac{\partial df}{\partial\om^\mu},\quad\mu=1,\dots,5.\label{parm}\ee
Now we consider the most general forms $(\theta^1,\theta^2,\theta^3,\theta^4,\theta^5)$ defining the same para-CR structure:
\be\bma
\theta^1\\\theta^2\\\theta^3\\\theta^4\\\theta^5\ema =\bma f^1&0&0&0&0\\ f^2&\rho{\rm e}^{\phi}&f^4&0&0\\ f^5&f^6&f^7&0&0\\ \bar{f}{}^2&0&0&\rho{\rm e}^{-\phi}&\bar{f}{}^4\\\bar{f}{}^5&0&0&\bar{f}{}^6&\bar{f}{}^7\ema\bma\om^1\\\om^2\\\om^3\\\om^4\\\om^5\ema,\label{mogi}\ee
and view them as \emph{lifted 1-forms} on the bundle $M\times G_0\to M$, where \be\begin{aligned}
  G_0=\Big\{\glg(5,\bbR)\ni S=&\bma f^1&0&0&0&0\\ f^2&\rho{\rm e}^{\phi}&f^4&0&0\\ f^5&f^6&f^7&0&0\\ \bar{f}{}^2&0&0&\rho{\rm e}^{-\phi}&\bar{f}{}^4\\\bar{f}{}^5&0&0&\bar{f}{}^6&\bar{f}{}^7\ema,\,\,\mathrm{with}\\
  &f^1,f^2,\bar{f}{}^2,f^4,\bar{f}{}^4,f^5,\bar{f}^5,f^6,\bar{f}{}^6,f^7,\bar{f}{}^7,\phi\in\bbR,\,\rho>0\Big\}.
\end{aligned}\label{grG0}\ee
We force the lifted 1-forms $(\theta^1,\theta^2,\theta^3,\theta^4,\theta^5)$ to satisfy a nonzero curvature version of equations \eqref{edsf}. In particular we want the forms $(\theta^1,\dots,\theta^5)$ to satisfy the first five of these equations:
\be
\begin{aligned}
E^1&=\der \theta^1-\Big(\theta^2\dz\theta^4-\theta^1\dz\Om_1\Big)-\sum_{1\leq i<j\leq 5}t^1{}_{ij}\theta^i\dz\theta^j\\
E^2&=\der \theta^2-\Big(\theta^3\dz\theta^4+\theta^2\dz(\Om_2-\tfrac12\Om_1)-\theta^1\dz\Om_3\Big)-\sum_{1\leq i<j\leq 5}t^2{}_{ij}\theta^i\dz\theta^j\\
E^3&=\der \theta^3-\Big(2\theta^3\dz\Om_2-\theta^2\dz\Om_3\Big)-\sum_{1\leq i<j\leq 5}t^3{}_{ij}\theta^i\dz\theta^j\\
E^4&=\der \theta^4-\Big(-\theta^2\dz\theta^5-\theta^4\dz(\tfrac12\Om_1+\Om_2)-\theta^1\dz\Om_4\Big)-\sum_{1\leq i<j\leq 5}t^4{}_{ij}\theta^i\dz\theta^j\\
E^5&=\der \theta^5-\Big(-2\theta^5\dz\Om_2+\theta^4\dz\Om_4\Big)-\sum_{1\leq i<j\leq 5}t^5{}_{ij}\theta^i\dz\theta^j,
\end{aligned}
\label{norme1}\ee
with `torsions' $t^i{}_{jk}$ \emph{as `minimal' as possible}. Although an ultimate goal would be to find a unique way of \emph{normalizing} these torsions in such a way that the resulting system for the forms $(\theta^i,\Om_\mu)$ describes a \emph{curvature of an} $\soa(2,3)$ \emph{Cartan connection}, we are not that ambitious here. Our aim is to find \emph{all pairwise locally nonequivalent homogeneous models} for these para-CR structures, so we are happy with any set of normalizations allowing to achieve this task. 

Actually, in the following we will require that the forms  $(\theta^i,\Om_\mu)$ should be linearly independent at each point, and that they should satisfy equations \eqref{norme} with $t^1{}_{ij}=t^2_{ij}=t^4{}_{ij}=0$ for all $i,j=1,2\dots,5$, and $t^3{}_{12}=t^3{}_{15}=t^3{}_{24}=t^3{}_{25}=0$, and $t^3{}_{ij}=0$ for all $3\leq i<j\leq 5$, as well as $t^5{}_{13}=t^5_{14}=t^5{}_{ij}=0$, for all $2\leq i<j\leq 5$ with an exception of $t^5{}_{45}\neq 0$. This means that we will require that our invariant forms will satisfy the following restricted form of equations \eqref{norme1}:
\be
\begin{aligned}
E^1&=\der \theta^1-\Big(\theta^2\dz\theta^4-\theta^1\dz\Om_1\Big)\\
E^2&=\der \theta^2-\Big(\theta^3\dz\theta^4+\theta^2\dz(\Om_2-\tfrac12\Om_1)-\theta^1\dz\Om_3\Big)\\
E^3&=\der \theta^3-\Big(2\theta^3\dz\Om_2-\theta^2\dz\Om_3\Big)-t^3{}_{13}\theta^1\dz\theta^3-t^3{}_{14}\theta^1\dz\theta^4-t^3{}_{23}\theta^2\dz\theta^3\\
E^4&=\der \theta^4-\Big(-\theta^2\dz\theta^5-\theta^4\dz(\tfrac12\Om_1+\Om_2)-\theta^1\dz\Om_4\Big)\\
E^5&=\der \theta^5-\Big(-2\theta^5\dz\Om_2+\theta^4\dz\Om_4\Big)-t^5{}_{12}\theta^1\dz\theta^2-t^3{}_{13}\theta^1\dz\theta^5-t^5{}_{45}\theta^4\dz\theta^5.
\end{aligned}
\label{norme}\ee
Note that we additionally require an equality of the coefficients at $\theta^1\dz\theta^3$ in $\der\theta^3$ and at $\theta^1\dz\theta^5$ in $\der\theta^5$.

We have the following theorem.
\begin{theorem}\label{the21}
  The torsion normalizations equations \eqref{norme} define the forms $(\theta^1,\theta^2,\theta^3,$ $\theta^4,\theta^5)$ as
$$\bma
  \theta^1\\\theta^2\\\theta^3\\\theta^4\\\theta^5\ema =S\cdot\bma -1&0&0&0&0\\0&1&0&0&0\\\frac{(2H_r^2+9H_p-3DH_r)}{18}&\frac{H_r}{3}&-1&0&0\\0&0&0&1&F_p\\\frac{3F_{pp}F_{pppp}-5F_{ppp}^2}{18F_{pp}^2}&0&0&\frac{F_{ppp}}{3F_{pp}}&\frac{F_{ppp}F_p-3F_{pp}^2}{3F_{pp}} \ema\cdot\bma\om^1\\\om^2\\\om^3\\\om^4\\\om^5\ema,$$
  with the matrix $S$ given by
  \be
  S=\bma \rho^2&0&0&0&0\\ f^2&\rho{\rm e}^\phi&0&0&0\\\frac{(f^2)^2}{2\rho^2}&\frac{f^2{\rm e}^\phi}{\rho}&{\rm e}^{2\phi}&0&0\\\bar{f}{}^2&0&0&\rho{\rm e}^{-\phi}&0\\ -\frac{(\bar{f}{}^2)^2}{2\rho^2}&0&0&\frac{-\bar{f}{}^2{\rm e}^{-\phi}}{\rho}&{\rm e}^{-2\phi}\ema.
  \label{matris}\ee
  The nonvanishing torsions $t^3{}_{14}$, $t^3{}_{23}$, $t^5{}_{12}$, $t^5{}_{45}$ read:
$$
\begin{aligned}
  t^3{}_{14}=\tfrac{{\rm e}^{3\phi}}{27\rho^3}A,\\
  t^5{}_{12}=\tfrac{{\rm e}^{-3\phi}}{27\rho^3}B,\\
  t^3{}_{23}=\tfrac{{\rm e}^{-\phi}}{3\rho}C,\\
  t^5{}_{45}=\tfrac{{\rm e}^{\phi}}{3\rho}\tilde{C},
  \end{aligned}
$$
where
$$
\begin{aligned}
  A&=-\tfrac12(9D^2H_r-27DH_p-18DH_r H_r+18H_p H_r+4H_r^3+54H_z),\\
  B&=\frac{40F_{ppp}^3-45 F_{pp}F_{ppp}F_{pppp}+9F_{pp}^2F_{ppppp}}{2F_{pp}^3},\\
  C&=\frac{2F_{ppp}+F_{pp}H_{rr}}{F_{pp}},\\
  \tilde{C}&=\frac{(DF_p+F_z)C-F_pC_4+C_5}{2F_{pp}}.
  \end{aligned}
$$ 
The vanishing or not of each of the quantities $A$, $B$, $C$, $\tilde{C}$ is an invariant property of the corresponding para-CR structure. 

The forms $\Om_1$, $\Om_2$, $\Om_3$, $\Om_4$ are given explicitly in terms of the defining functions of the para-CR structure, their derivatives, fiber variables $(\rho,\phi,f_2,\bar{f}{}_2)$, and one new real variable, which we call $u_1$.   
\end{theorem}
\begin{proof}
We use the normalization equations \eqref{norme}.
  
  We first impose equation
  $$0=E^1\dz\theta^1=(\der\theta^1-\theta^2\dz\theta^4)\dz\theta^1.$$
  This immediately gives
  \be
  f^4=0,\quad\quad\bar{f}{}^4=F_p\rho\mathrm{e}^{-\phi},\quad\quad f^1=-\rho^2\label{f4}\ee
 in \eqref{mogi}. Then we go to impose
 $$
 \begin{aligned}
  0&=E^2\dz\theta^1\dz\theta^2=(\der\theta^2-\theta^3\dz\theta^4)\dz\theta^1\dz\theta^2,\\
  0&=E^4\dz\theta^1\dz\theta^4=(\der\theta^4-\theta^5\dz\theta^2)\dz\theta^1\dz\theta^4,
 \end{aligned}
 $$
 which additionally gives
 \be
 f^7=-{\rm e}^{2\phi},\quad\quad\&\quad\quad \bar{f}{}^7=-F_{pp}{\rm e}^{-2\phi}+F_p\bar{f}{}^6\label{f7}\ee
 in \eqref{mogi}.

 After these normalizations we have \eqref{mogi} in the form
 $$\bma
\theta^1\\\theta^2\\\theta^3\\\theta^4\\\theta^5\ema =\bma -\rho^2&0&0&0&0\\ f^2&\rho{\rm e}^{\phi}&0&0&0\\ f^5&f^6&-{\rm e}^{2\phi}&0&0\\ \bar{f}{}^2&0&0&\rho{\rm e}^{-\phi}&\rho\mathrm{e}^{-\phi}F_p\\\bar{f}{}^5&0&0&\bar{f}{}^6&-{\rm e}^{-2\phi}F_{pp}+\bar{f}{}^6F_p\ema\bma\om^1\\\om^2\\\om^3\\\om^4\\\om^5\ema.
$$
Now we impose the first equation
$$0=E^1=\der\theta^1-(\theta^2\dz\theta^4-\theta^1\dz\Omega_1).$$
This defines the 1-form $\Omega_1$ as:
\be\Omega_1=2\der\log(\rho)-u_1\theta^1+\frac{\bar{f}{}^2}{\rho^2}\theta^2+(-\frac{f^2}{\rho^2}+\frac{\bar{f}{}^6F_z{\rm e}^{3\phi}F_z}{\rho F_{pp}})\theta^4-\frac{{\rm e}^{2\phi}F_z}{F_{pp}}\theta^5.\label{u1}\ee
Note that to define this form we needed to introduce a \emph{new} variable
$u_1$.

Thus we have normalized our forms $(\theta^i,\Omega_\mu)$ in such a way that
$$\der\theta^1=\theta^2\dz\theta^4-\theta^1\dz\Omega_1.$$

Let us pass to the equations $E^2=0$ and $E^4=0$. We first impose
$$E^2\dz\theta^1=E^4\dz\theta^1=0.$$
It is easy to check that this is equivalent to
$$\begin{aligned}\Omega_2&=-\der\phi+u_2\theta^1+\frac{2\rho{\rm e}^\phi\bar{f}{}^6-3\bar{f}{}^2}{2\rho^2}\theta^2+\frac{{\rm e}^{2\phi}(2DF_p+F_z)}{2F_{pp}}\theta^5+\\&-\frac{3f^2F_{pp}+\rho\mathrm{e}^{3\phi}\bar{f}{}^6(2DF_p+F_z)+2\rho\mathrm{e}^{-\phi}F_{pp}}{2F_{pp}\rho^2}\theta^4.
\end{aligned}
$$
Now it is easy to solve for $\Omega_3$ and $\Omega_4$ from
$$0=E^2=E^4.$$
This gives
$$\begin{aligned}
  \Omega_3&=-\frac{\der f_2}{\rho^2}+\dots-u_3 \theta^1\\
  \Omega_4&=-\frac{\der \bar{f}{}_2}{\rho^2}+\dots-u_4 \theta^1,\end{aligned}$$
where we have indicated that both of these forms are given modulo an addition of a term proportional to $\theta^1$ by introducing new variables $u_3$ and $u_4$.

Next we impose
$$0=E^3\dz\theta^1\dz\theta^2=E^5\dz\theta^1\dz\theta^4.$$
This gives
\be
f^6=-\frac{f^2{\rm e}^\phi}{\rho}+\tfrac13{\rm e}^{2\phi}H_r\quad\quad\&\quad\quad\bar{f}{}^6=\frac{\bar{f}{}^2{\rm e}^{-\phi}}{\rho}+\frac{F_{ppp}{\rm e}^{-2\phi}}{3F_{pp}}\label{f6}\ee
in\eqref{mogi}.

Now,
$$0=E^3\dz\theta^1=E^5\dz\theta^1,$$
gives $t^3{}_{23}$ and $t^5{}_{45}$ precisely as in the thesis of the theorem. It also gives that:
\be f^5=-\frac{(f^2)^2}{2\rho^2}-\frac{1}{18}\mathrm{e}^{2\phi}(3DH_r-9H_p-2H_r^2)\quad\&\quad\bar{f}{}^5=\frac{(\bar{f}{}^2)^2}{2\rho^2}-\frac{1}{18}\mathrm{e}^{-2\phi}\frac{5F_{ppp}^2-3F_{pp}F_{pppp}}{F_{pp}^2}.\label{f5}\ee

Finally,
$$0=E^3=E^5$$
gives $t^3{}_{14}$ and $t^5{}_{12}$ precisely as in the thesis of the theorem. It also gives an explicit formula for
$t^3{}_{13}$, $u_2$, $u_3$ and $u_4$. In particular,
\be
t^3{}_{13}=\frac{f^2 C\mathrm{e}^{-\phi}}{6\rho^3}+\frac{\bar{f}{}^2 \tilde{C}\mathrm{e}^{\phi}}{6\rho^3}-\frac{T}{\rho^2},\label{TQ}\ee
with a function $T$ on $M$ having a property that it vanishes when $C\equiv 0$. It is also worth mentioning that in the obtained formulas $u_2$ depends on the para-CR structure and the variables  $(\rho,\phi,f^2,\bar{f}{}^2)$, and $u_3$ and $u_4$ depend on the para-CR structure and the variables  $(\rho,\phi,f_2,\bar{f}{}_2,u_1)$.

Note that with the normalizations \eqref{f4}, \eqref{f7}, \eqref{f6} and \eqref{f5}, the matrix $S$ bringing the forms $(\om^1,\om^2,\om^3,\om^4,\om^5)$ to $(\theta^1,\theta^2,\theta^3,\theta^4,\theta^5)$ becomes precisely as in the thesis of the theorem, provided that we change fiber coordinates according to $(f^2,\bar{f}{}^2)\to (-f^2,-\bar{f}{}^2)$.

The relation between relative invariant $\tilde{C}$ and $C$ together with its coframe derivatives, is a consequence of integrablity conditions ($\der^2\equiv 0$) for the system \eqref{iniomi} and, in particular, the condition that $\der^2H_r\equiv 0$.

This finishes the proof.
\end{proof}

An immediate consequence of this theorem is the following
\begin{corollary}\label{sysc}
It is always possible to force the lifted coframe $(\theta^1,\theta^2,\theta^3,\theta^4,\theta^5)$ of a PDE five variables para-CR structure to satisfy the following exterior differential system: 
\be
\begin{aligned}
\der \theta^1=&-\theta^1\dz\Om_1+\theta^2\dz\theta^4\\
\der \theta^2=&\,\theta^2\dz(\Om_2-\tfrac12\Om_1)-\theta^1\dz\Om_3+\theta^3\dz\theta^4\\
\der \theta^3=&\,2\theta^3\dz\Om_2-\theta^2\dz\Om_3+Q\,\theta^1\dz\theta^3+\tfrac{{\rm e}^{3\phi}}{27\rho^3}A\,\theta^1\dz\theta^4+\tfrac{{\rm e}^{-\phi}}{3\rho}C\,\theta^2\dz\theta^3\\
\der \theta^4=&-\theta^2\dz\theta^5-\theta^4\dz(\tfrac12\Om_1+\Om_2)-\theta^1\dz\Om_4\\
\der \theta^5=&-2\theta^5\dz\Om_2+\theta^4\dz\Om_4+\tfrac{{\rm e}^{-3\phi}}{27\rho^3}B\,\theta^1\dz\theta^2+Q\,\theta^1\dz\theta^5+\tfrac{{\rm e}^{\phi}}{3\rho}\tilde{C}\,\theta^4\dz\theta^5.
\end{aligned}
\label{normuj}\ee
Here, the functions $A,B,C$ and $\tilde{C}$ are functions on the base manifold $M$, where the para-CR structure is defined, and are obtained in terms of the functions $F$ and $H$ defining the para-CR structure and their derivatives, i.e. they do not depend on the fiber coordinates $(f^2,\bar{f}{}^2,\rho,\phi)$. The function $Q$ depends on fiber coordinates as $t^3{}_{13}$ in \eqref{TQ}.\qed
\end{corollary}

\section{Cartan's reduction: homogeneous models}

\subsection{New notation and the relative invariants}
Corollary \ref{sysc} assures that by means of transformations \eqref{mogi} we can always bring the initial 1-forms $(\om^1,\om^2,\om^3,$ $\om^4,\om^5)$ defining our PDE five variables para-CR structure via \eqref{iniom}, into an equivalent set of 1-forms on a para-CR manifold $M$ satisfying the following EDS:
\be
\begin{aligned}
\der \om^1=&-\om^1\dz\varpi_1+\om^2\dz\om^4\\
\der \om^2=&-\om^1\dz\varpi_3+\om^2\dz(\varpi_2-\tfrac12\varpi_1)+\om^3\dz\om^4\\
\der \om^3=&-\om^2\dz\varpi_3+2\om^3\dz\varpi_2+V\,\,\om^1\dz\om^3+I^1\,\om^1\dz\om^4+I^3\,\om^2\dz\om^3\\
\der \om^4=&-\om^1\dz\varpi_4-\om^4\dz(\varpi_2+\tfrac12\varpi_1)-\om^2\dz\om^5\\
\der \om^5=&\,\om^4\dz\varpi_4-2\om^5\dz\varpi_2+I^2\,\om^1\dz\om^2+V\,\,\om^1\dz\om^5+I^4\,\om^4\dz\om^5,
\end{aligned}\label{systart}
\ee
with some set of auxiliary forms $(\varpi_1,\varpi_2,\varpi_3,\varpi_4)$ and certain functions $V$, $I^1$, $I^2$, $I^3$ and $I^4$ on $M$. For this it is enough to take a section $\sigma:=\big(f^2=0,\bar{f}{}^2=0,\rho=\tfrac13,\phi=0, u_1=0\big)$ of this bundle over $M$ which is described in Corollary \ref{sysc}, and to take the $\sigma$ pullbacks of the forms $(\theta^1,\theta^2,\theta^3,\theta^4,\theta^5,\Om_1,\Om_2,\Om_3,\Om_4)$ as the forms $(\om^1,\om^2,\om^3,\om^4,\om^5,\varpi_1,\varpi_2,\varpi_3,\varpi_4)$ in the EDS \eqref{systart}. In terms of this pullback
the functions $I^1,I^2,I^3,I^4$ become the respective relative invariants $A,B,C,\tilde{C}$ of the considered para-CR structure.

Due to Corollary \ref{sysc} the system \eqref{systart} can be also interpreted as the structural EDS for any PDE five variables para-CR structure on a \emph{nine-dimensional bundle} ${\mathcal G}^9\to M$ \emph{over the para-CR manifold} $M$, where all the nine 1-forms $(\om^1,\om^2,\om^3,\om^4,$ $\om^5,\varpi_1,\varpi_2,\varpi_3,\varpi_4)$ are linearly independent, and the functions $V,I^1,I^2,I^3,I^4$ are functions on the bundle. For this think about $(\om^1,\om^2,\om^3,$ $\om^4,\om^5,\varpi_1,\varpi_2,\varpi_3,\varpi_4)$ as $(\theta^1,\theta^2,\theta^3,\theta^4,\theta^5,\Om_1,\Om_2,\Om_3,\Om_4)$, and of $V,I^1,I^2,I^3,I^4$ as $Q$, $\tfrac{{\rm e}^{3\phi}}{27\rho^3}A$, $\tfrac{{\rm e}^{-3\phi}}{27\rho^3}B$, $\tfrac{{\rm e}^{-\phi}}{3\rho}C$, $\tfrac{{\rm e}^{\phi}}{3\rho}\tilde{C}$ in \eqref{normuj}, respectively.

In the following we will use the para-CR structural system \eqref{systart} having in mind both of the above interpretations. Essentially every argument we will give can be interpreted either in the first or in the second way. It is a matter of convenience to choose one of them. For example, in this section, we will adopt the following \emph{notation}:

Consider a differentiable function $f$ as a function on bundle ${\mathcal G}^9$, i.e. use the second interpretation. Its differential decomposes onto the \emph{basis} of the 1-forms $(\om^1,\om^2,\om^3,\om^4,$ $\om^5,\varpi_1,\varpi_2,\varpi_3,\varpi_4)$ on ${\mathcal G}^9$ and we \emph{denote} the coeffcients of this decomposition as:
\be
\der f=f_{|1}\om^1+f_{|2}\om^2+f_{|3}\om_3+f_{|4}\om^4+f_{|5}\om^5+(\dots)\varpi_1+(\dots)\varpi_2+(\dots)\varpi_3+(\dots)\varpi_4.\label{notder}\ee
We especially do not assign particular notation to the dotted coefficients, because once the notation for the coeffcients at $\om^\mu$ is set, the dotted coeffcients follow from $\der^2=0$ applied to the system \eqref{systart} and to $f$.

Thus, in the bundle ${\mathcal G}^9$ interpretation, the symbol $f_{|\mu}$ denotes the directional derivative of $f$ in the direction of the vector field $X_\mu$, which constitutes the $\mu$ component of the basis of vector fields $(X_1,X_2,X_3,X_4,X_5,Y^1,Y^2,Y^3,Y^4)$ dual to the coframe $(\om^1,\om^2,\om^3,\om^4\om^5,\Omega_1,\Omega_2,\Omega_3,\Omega_4)$ on ${\mathcal G}^9$.

On the other hand, if we interpret the system \eqref{systart} directly on $M$, the introduced notation \eqref{notder} for $\der f$, does not mean anymore that e.g. $f_{|3}$ is the coframe $\om^3$ derivative $f_3$ of $f$, but that it is a \emph{corrected} coframe $\om^3$ derivative by the terms coming from the dotted coeffcients standing at $\om^3$ parts of $\varpi_i$s. This happens because now the forms $\varpi_i$ are linearly dependent on $\om^\mu$s.

Anyhow, it turns out that when compared with the notation for $\der f$ introduced in \eqref{parm}, the new notation, as in \eqref{notder}, considerably simplifies the formulas we are going to derive in this section.

So now, having both interpretations of the system \eqref{systart} in mind, and knowing that it encodes all the structural information about an arbitrary PDE five variables para-CR structure, we consider it as an \emph{abstract} exterior differential system, and we will \emph{close it}, namely apply the condition $\der^2\equiv 0$, as far as it is needed for our purposes, obtaining in particular information on the derivatives of the structural functions $I^1,I^2,I^3,I^4$ and $V$.

This leads to the following statement.
\begin{proposition}\label{susend}
  The differential consequences of the system \eqref{systart} are:
  \be
\begin{aligned}
\der \om^1=&-\om^1\dz\varpi_1+\om^2\dz\om^4\\
\der \om^2=&-\om^1\dz\varpi_3+\om^2\dz(\varpi_2-\tfrac12\varpi_1)+\om^3\dz\om^4\\
\der \om^3=&-\om^2\dz\varpi_3+2\om^3\dz\varpi_2+\tfrac18(2I^3{}_{|4}+I^3{}_{|52})\om^1\dz\om^3+\\&I^1\,\om^1\dz\om^4+I^3\,\om^2\dz\om^3\\
\der \om^4=&-\om^1\dz\varpi_4-\om^4\dz(\varpi_2+\tfrac12\varpi_1)-\om^2\dz\om^5\\
\der \om^5=&\,\om^4\dz\varpi_4-2\om^5\dz\varpi_2+I^2\,\om^1\dz\om^2+\tfrac18(2I^3{}_{|4}+I^3{}_{|52})\om^1\dz\om^5-\\&\tfrac12I^3{}_{|5}\,\om^4\dz\om^5.
\end{aligned}\label{sysend}
\ee
$$\begin{aligned}
\der I^1=&\,I^1{}_{|1}\om^1+I^1{}_{|2}\om^2+I^1{}_{|3}\om^3+I^1{}_{|4}\om^4-\tfrac32 I^1\varpi_1-3I^1\varpi_2\\
\der I^2=&\,I^2{}_{|1}\om^1+I^2{}_{|2}\om^2+I^2{}_{|4}\om^4+I^2{}_{|5}\om^5-\tfrac32 I^2\varpi_1+3I^2\varpi_2\\
\der I^3=&\,I^3{}_{|1}\om^1+I^3{}_{|2}\om^2+I^3{}_{|3}\om^3+I^3{}_{|4}\om^4+I^3{}_{|5}\om^5-\tfrac12 I^3\varpi_1+ I^3\varpi_2,\end{aligned}$$
 \be
\begin{aligned}
\der \varpi_1=&\,\om^1\dz\varpi_5+\om^2\dz\varpi_4-\om^4\dz\varpi_3\\
\der \varpi_2=&-\tfrac14 I^3\om^1\dz\varpi_3-\tfrac18I^3{}_{|5}\om^1\dz\varpi_4-\tfrac12\om^2\dz\varpi_4-\tfrac12\om^4\dz\varpi_3+\\&\tfrac{1}{16}(I^3{}_{|522}+2I^3{}_{|42}-8I^2{}_{|5})\om^1\dz\om^2+\tfrac{1}{16}(I^3{}_{|523}+2I^3{}_{|43})\om^1\dz\om^3+\\&\tfrac{1}{16}(8I^1{}_{|3}-I^3{}_{|524}-2I^3{}_{|44})\om^1\dz\om^4-\tfrac{1}{16}(I^3{}_{|525}+2I^3{}_{|45})\om^1\dz\om^5+\\&\tfrac{1}{8}(I^3{}_{|52}-2I^3{}_{|4})\om^2\dz\om^4-\tfrac12I^3{}_{|5}\om^2\dz\om^5+I^3\om^3\dz\om^4-\om^3\dz\om^5\\
\der \varpi_3=&\,\varpi_3\dz(\tfrac12\varpi_1+\varpi_2)+\tfrac18(2I^3{}_{|4}+I^3{}_{|52})\om^1\dz\varpi_3+\tfrac14 I^3\om^2\dz\varpi_3+\\&\tfrac18I^3{}_{|5}\om^2\dz\varpi_4+\tfrac12\om^2\dz\varpi_5+\om^3\dz\varpi_4+J^1\om^1\dz\om^2+\\&\tfrac{1}{4}(4I^2{}_{|5}+4I^3{}_{|1}-2I^3{}_{|42}-I^3{}_{|522})\om^1\dz\om^3+(I^1I^3-I^1{}_{|2})\om^1\dz\om^4+\\&I^1\om^1\dz\om^5-\tfrac{1}{16}(2I^3{}_{|43}+I^3{}_{|523})\om^2\dz\om^3+\tfrac{1}{16}(I^3{}_{|524}-8I^1{}_{|3}+2I^3{}_{|44})\om^2\dz\om^4+\\&\tfrac{1}{16}(2I^3{}_{|45}+I^3{}_{|525})\om^2\dz\om^5-\tfrac18(2I^3{}_{|4}+I^3{}_{|52})\om^3\dz\om^4\\
\der \varpi_4=&\,\varpi_4\dz(\tfrac12\varpi_1-\varpi_2)+\tfrac18(2I^3{}_{|4}+I^3{}_{|52})\om^1\dz\varpi_4-\tfrac14 I^3\om^4\dz\varpi_3-\\&\tfrac18I^3{}_{|5}\om^4\dz\varpi_4+\tfrac12\om^4\dz\varpi_5+\om^5\dz\varpi_3+\\&\tfrac{1}{128}\Big(16(I^3{}_{|14}-I^1I^3{}_{|3})+8(I^3{}_{|521}-I^1{}_{|3}I^3)+2I^3I^3{}_{|44}+I^3I^3{}_{|524}\Big)\om^1\dz\om^4+\\&\tfrac{1}{2}(2I^2{}_{|4}+I^2I^3{}_{|5})\om^1\dz\om^2-I^2\om^1\dz\om^3+\tfrac{1}{16}(8I^2{}_{|5}-2I^3{}_{|42}-I^3{}_{|522})\om^2\dz\om^4\\&+\tfrac{1}{4}(I^3{}_{|524}-4I^1{}_{|3}+2I^3{}_{|44}+2I^3{}_{|51})\om^1\dz\om^5+\tfrac{1}{8}(2I^3{}_{|4}+I^3{}_{|52})\om^2\dz\om^5-\\&\tfrac{1}{16}(2I^3{}_{|43}+I^3{}_{|523})\om^3\dz\om^4-\tfrac{1}{16}(2I^3{}_{|45}+I^3{}_{|525})\om^4\dz\om^5\\
\der \varpi_5=&\,\varpi_5\dz\varpi_1+2\varpi_4\dz\varpi_3+J^2\om^1\dz\varpi_3+J^3\om^1\dz\varpi_4+\tfrac14(2I^3{}_{|4}+I^3{}_{|52})\om^1\dz\varpi_5+\\&\tfrac18(2I^3{}_{|4}+I^3{}_{|52})\om^2\dz\varpi_4-\tfrac18(2I^3{}_{|4}+I^3{}_{|52})\om^4\dz\varpi_4+J^4\om^1\dz\om^2+J^5\om^1\dz\om^3+\\&J^6 \om^1\dz\om^4+J^7\om^1\dz\om^5-I^2\om^2\dz\om^3+J^8\om^2\dz\om^4+\\&\tfrac{1}{4}(I^3{}_{|524}-4I^1{}_{|3}+2I^3{}_{|44}+2I^3{}_{|51})\om^2\dz\om^5+\\&\tfrac{1}{4}(4I^2{}_{|5}+4I^3{}_{|1}-2I^3{}_{|42}-I^3{}_{|522})\om^3\dz\om^4-I^1\om^4\dz\om^5.
\end{aligned}
\label{sysend1}
\ee
$$\begin{aligned}\der I^3{}_{|2}=&\,\tfrac{1}{16}\Big(16(I^3{}_{|12}-I^2I^3{}_{|5})+I^3(8I^2{}_{|5}-2I^3{}_{|42}-I^3{}_{|522})\Big)\om^1+I^3{}_{|22}\om^2+I^3{}_{|23}\om^3+\\&
\tfrac18\Big(8(I^3{}_{|42}+I^3{}_{|1})+I^3(I^3{}_{|52}-2I^3{}_{|4})\Big)\om^4+\tfrac12\Big(2(I^3{}_{|52}-I^3{}_{|4})-I^3I^3{}_{|5}\Big)\om^5-\\&
I^3{}_{|2}\varpi_1+2I^3{}_{|2}\varpi_1-I^3{}_{|3}\varpi_3-I^3\varpi_4\\
\der I^3{}_{|3}=&\,\tfrac{1}{16}\Big(16I^3{}_{|13}-2I^3{}_{|3}(2I^3{}_{|4}+I^3{}_{|52})-I^3(I^3{}_{|523}+2I^3{}_{|43})\Big)\om^1+(I^3{}_{|23}-I^3I^3{}_{|3})\om^2+\\&I^3{}_{|33}\om^3+
\tfrac12\Big(I^3{}_{|523}+2I^3{}_{|43}-2(I^3{}_{|2}+(I^3)^2)\Big)\om^4+3I^3\om^5-
\tfrac12I^3{}_{|3}\varpi_1+3I^3{}_{|3}\varpi_2\\
\der I^3{}_{|5}=&\,I^3{}_{|51}\om^1+I^3{}_{|52}\om^2+4I^3\om^3+I^3{}_{|54}\om^4+I^3{}_{|55}\om^5-\tfrac12I^3{}_{|5}\varpi_1-I^3{}_{|5}\varpi_2\\
\der I^3{}_{|52}=&\,I^3{}_{|521}\om^1+I^3{}_{|522}\om^2+4\big((I^3)^2+I^3{}_{|2}\big)\om^3+I^3{}_{|524}\om^4+\\&\big(2I^3{}_{|45}+(I^3{}_{|5})^2+I^3{}_{|525}-2I^3{}_{|54}\big)\om^5-I^3{}_{|52}\varpi_1-4I^3\varpi_3.
\end{aligned}
$$
The coefficients $J^1,J^2,\dots,J^8$ are not important here.
  \end{proposition}

We now use this proposition interpreting forms $(\om^1,\om^2,\om^3,\om^4,\om^5)$ as defining a specially adapted coframe of an arbitrary PDE five variables para-CR structure, and use it to build the lifted coframe \eqref{mogi} which satisfies equations of the form \eqref{norme1}. This in turn, by the same procedure which we used to get Theorem \ref{the21}, leads to the reinterpretation of this Theorem and Corollary \ref{sysc} into the following form:
\begin{corollary}\label{nurhuj}
  The torsion normalizations equations \eqref{norme} applied to the forms \eqref{mogi} with $(\om^1,\om^2,\om^3,\om^4,\om^5)$ satisfying \eqref{sysend} yield the following para-CR invariant differential system
 \be
\begin{aligned}
\der \theta^1=&-\theta^1\dz\Om_1+\theta^2\dz\theta^4\\
\der \theta^2=&\theta^2\dz(\Om_2-\tfrac12\Om_1)-\theta^1\dz\Om_3+\theta^3\dz\theta^4\\
\der \theta^3=&2\theta^3\dz\Om_2-\theta^2\dz\Om_3+\tfrac{{\rm e}^{3\phi}}{\rho^3}I^1\,\theta^1\dz\theta^4+\tfrac{{\rm e}^{-\phi}}{\rho}I^3\,\theta^2\dz\theta^3+\\&\tfrac{1}{8\rho^3}\,\Big(2\mathrm{e}^\phi\bar{f}{}^2I^3{}_{|5}+\rho(I^3{}_{|52}+2I^3{}_{|4})-4\mathrm{e}^{-\phi}f^2I^3\Big)\,\theta^1\dz\theta^3\\
\der \theta^4=&-\theta^2\dz\theta^5-\theta^4\dz(\tfrac12\Om_1+\Om_2)-\theta^1\dz\Om_4\\
\der \theta^5=&-2\theta^5\dz\Om_2+\theta^4\dz\Om_4+\tfrac{{\rm e}^{-3\phi}}{\rho^3}I^2\,\theta^1\dz\theta^2-\tfrac{{\rm e}^{\phi}}{2\rho}I^3{}_{|5}\,\theta^4\dz\theta^5+\\&\tfrac{1}{8\rho^3}\Big(2\mathrm{e}^\phi\bar{f}{}^2I^3{}_{|5}+\rho(I^3{}_{|52}+2I^3{}_{|4})-4\mathrm{e}^{-\phi}f^2I^3\Big)\,\theta^1\dz\theta^5.
\end{aligned}
\label{norhuj}\ee
\end{corollary}
\begin{proof}
  Since this Corollary is just a reformulation, in the new notation, of Theorem \ref{the21} we only give the matrix $S=(S^\mu{}_\nu)$ which, via $\theta^\mu=S^\mu{}_\nu\om^\nu$, brings the system \eqref{sysend} to the system \eqref{norhuj}. It is not a surprise that this matrix is precisely given by the formula \eqref{matris}.
\end{proof}
This shows that in the notation of this section the simplest \emph{relative invariants} of the considered para-CR structures are $I^1,I^2,I^3$ and $I^3{}_{|5}$. In particular the structures with the structural function $I^3\neq 0$ and $I^3=0$ are locally para-CR nonequivalent.
\subsection{The case $I^3\neq 0$ and corresponding homogeneous models}
If the relative invariant $I^3\neq 0$ we can normalize the term at $\theta^2\dz\theta^3$ in $\der\theta^3$ to 1,
$$\frac{\mathrm{e}^{-\phi}I^3}{\rho}=1,$$
reducing the system \eqref{norhuj} by one dimension, due to the choice of the section
$$\rho=\mathrm{e}^{-\phi}I^3.$$
After this normalization the form $\Om_2$ becomes dependent on the forms $\Om_1$, $\theta^\mu$, and thus it disappears from the equations \eqref{norhuj}. Actually, the entire combination $\Om_2-\tfrac12\Om_1$ disappears from these equations. This in particular gives
$$
  (\der \theta^2)\dz \theta^3\dz\theta^4\dz\theta^5=\Big(-\theta^1\dz\Om_3+\frac{K}{8(I^3)^5}\theta^1\dz\theta^2\Big)\dz\theta^3\dz\theta^4\dz\theta^5,
$$
with the coefficient $K=4(I^3)^5u_1+L$. Here $u_1$ is an auxiliary variable introduced when normalizing the system \eqref{norhuj}. It is analogous to $u_1$ introduced in \eqref{u1}. The quantity $L$ depends on the structural function $I^3$, its derivatives, and the free fiber coordinates $\phi$, $f^2$ and $\bar{f}{}^2$, only. The explicit linear $u_1$-dependence of $K$, where the $u_1$ term is multiplied by a fifth power of $I^3$, which is assumed not to vanish, enables us to normalize the coefficient at $\theta^1\dz\theta^2$ in $\der\theta^2$ to 0,
$$K=0.$$
This eliminates the auxiliary variable $u_1$ from the system.

After this normalization we get in paricular that:
$$\begin{aligned}\der \theta^4=&-\theta^2\dz\theta^5-\theta^4\dz\Om_1-\theta^1\dz\Om_4+\frac{\mathrm{e}^{2\phi}\bar{f}{}^2I^3+f^2I^3{}_{|3}-I^3I^3{}_{|2}}{(I^3)^3}\theta^2\dz\theta^4-\\&
\frac{\mathrm{e}^{-2\phi}I^3{}_{|3}}{I^3}\theta^3\dz\theta^4+\frac{\mathrm{e}^{2\phi}I^3{}_{|5}}{I^3}\theta^4\dz\theta^5.
\end{aligned}$$
This enables for further reduction, by forcing the coefficient of $\theta^2\dz\theta^4$ to vanish. This results in the restriction of the system \eqref{norhuj} to a section
$$\bar{f}^2=\frac{I^3I^3{}_{|2}-f^2I^3{}_{|3}}{\mathrm{e}^{2\phi}I^3},$$
on which the form $\Om_4$ becomes dependent on $\Om_3$ and $\theta^\mu$s. Thus, it is not present in the reduced system in which, in particular the differential of $\theta^2$ reads:
$$\begin{aligned}\der \theta^2=&\,\theta^3\dz\theta^4-\theta^1\dz\Om_3+\mathrm{e}^{2\phi}\big(\dots\big)\theta^2\dz\theta^4-
\frac{\mathrm{e}^{-2\phi}I^3{}_{|3}}{I^3}\theta^2\dz\theta^3-\frac{\mathrm{e}^{2\phi}I^3{}_{|5}}{I^3}\theta^2\dz\theta^5.
\end{aligned}$$
This shows that when $I^3\neq 0$, which we assume in this section, the structural functions $I^3{}_{|3}$ and $I^3{}_{|5}$ are \emph{relative invariants} of such para-CR structures. In particular, if we have two para-CR structures, one with $I^3{}_{|5}\neq 0$ and the other with $I^3{}_{|5}=0$, then they are locally para-CR nonequivalent.

Let us first concentrate on the case when
$$I^3{}_{|5}\neq 0.$$
In this case we can normalize the term at $\theta^2\dz\theta^5$ in $\der\theta^2$ to be equal to $-\epsilon$, where $$\epsilon=\sgn(\tfrac{I^3{}_{|5}}{I^3}).$$
This results in further reduction of the system \eqref{norhuj} to the section on which
$$\phi=\tfrac12\log(\tfrac{\epsilon I^3}{I^3{}_{|5}}).$$
This makes $\Om_1$ dependent on $\theta^\mu$s only, and eliminates $\Om_1$ from the variables in the reduced EDS. 
Then in particular, the differential of $\theta^1$ satisfies
$$
  (\der \theta^1)\dz \theta^3\dz\theta^4\dz\theta^5=\frac{4f^2-I^3{}_{|52}}{I^3I^3{}_{|5}}\theta^1\dz\theta^2\dz\theta^3\dz\theta^4\dz\theta^5.
$$
This enables for the ultimate normalization which kills the $\theta^1\dz\theta^2$ term in $\der\theta^1$. It is obtained by taking the section
$$f^2=\tfrac14 I^3{}_{|52}.$$
After this normalization the system \eqref{norhuj} reduces to the original five manifold $M$ on which the para-CR structure is defined. It brings the initial forms $\omega^\mu$ satisfying the system \eqref{sysend} to the fully para-CR invariant forms $\theta^\mu$ on $M$, via the formula $\theta^\mu=S^\mu{}_{\nu}\om^\nu$, where the matrix $S=(S^\mu{}_{\nu})$ is given by:
$$S={\tiny \bma
\epsilon I^3I^3{}_{|5}&0&0&0&0\\
\tfrac14 I^3{}_{|52}&I^3&0&0&0\\
\tfrac{\eps \big(I^3{}_{|52}\big)^2}{32 I^3 I^3{}_{|5}}&\tfrac{\eps I^3{}_{|352}}{4 I^3{}_{|5}}&\tfrac{\eps I^3}{I^3{}_{|5}}&0&0\\
\tfrac{I^3{}_{|5}}{4\epsilon \big(I^3\big)^2}\big(4I^3I^3{}_{|2}-I^3{}_{|3}I^3{}_{|52}\big)&0&0&\epsilon I^3{}_{|5}&0\\
-\tfrac{I^3{}_{|5}}{32\epsilon \big(I^3\big)^5}\big(4I^3I^3{}_{|2}-I^3{}_{|3}I^3{}_{|52}\big)^2&0&0&-\tfrac{I^3{}_{|5}}{4\epsilon \big(I^3\big)^3}\big(4I^3I^3{}_{|2}-I^3{}_{|3}I^3{}_{|52}\big)&\epsilon\tfrac{I^3{}_{|5}}{I^3}
\ema}.$$
The resulting para-CR invariant EDS on $M$ is presented in the following statement.
\begin{theorem}
  Every PDE five variables para-CR structure on a 5-dimensional manifold $M$ with the relative invariants $I^3\neq0$ and $I^3{}_{|5}\neq 0$ uniquely defines five 1-forms $(\theta^1,\theta^2,\theta^3,\theta^4,\theta^5)$ on $M$ which satisfy the following exterior differential system
  \be
  \begin{aligned}
  \der\theta^1=&\,\epsilon\Big(-c_3\theta^1\dz\theta^3+c_5\theta^1\dz\theta^4-c_4\theta^1\dz\theta^5\Big)+\theta^2\dz\theta^4\\
  \der\theta^2=&\,\epsilon\Big(c_8\theta^1\dz\theta^2+(4-c_3)\theta^2\dz\theta^3-c_9\theta^2\dz\theta^4-\theta^2\dz\theta^5\Big)-\theta^1\dz\theta^3+\\&c_7\theta^1\dz\theta^4-c_6\theta^1\dz\theta^5+\theta^3\dz\theta^4\\
  \der\theta^3=&\,\epsilon\Big(-c_{11}\theta^1\dz\theta^3+c_1\theta^1\dz\theta^4-(c_5+2c_9)\theta^3\dz\theta^4+(c_4-2)\theta^3\dz\theta^5\Big)+\\&c_{10}\theta^1\dz\theta^2+c_7\theta^2\dz\theta^4-c_6\theta^2\dz\theta^5\\
  \der\theta^4=&\,\epsilon\Big(c_{12}\theta^1\dz\theta^3+c_{14}\theta^1\dz\theta^4+\tfrac12(1-8c_6+2c_3c_6+2c_9)\theta^1\dz\theta^5+\\&4\theta^3\dz\theta^4+(1-c_4)\theta^4\dz\theta^5\Big)-c_{13}\theta^1\dz\theta^2-\theta^2\dz\theta^5\\
  \der\theta^5=&\,\epsilon\Big(c_{15}\theta^1\dz\theta^4+\tfrac12(c_9+2c_{11})\theta^1\dz\theta^5+c_{12}\theta^3\dz\theta^4+(8-c_3)\theta^3\dz\theta^5-\\&(1+c_5-4c_6+c_3c_6+3c_9)\theta^4\dz\theta^5\Big)+c_2\theta^1\dz\theta^2-c_{13}\theta^2\dz\theta^4.
  \end{aligned}\label{homo1}
  \ee
  Here, the coefficients $c_1,c_2,\dots,c_{15}$ are totally expressed in terms of the structural functions $I^1,I^2,I^3$ and their derivatives. They are uniquely defined by the para-CR structure on $M$ (but their explicit forms are not relevant here) .

  Two different PDE five variables para-CR structures with their corresponding relative invariants $I^3\neq0$ and $I^3{}_{|5}\neq 0$ are locally para-CR equivalent if and only if their corresponding 1-forms $\theta^\mu$ and $\bar{\theta}{}^\mu$  can be transformed to each other by a local diffeomorphism $\Phi$, i.e. if and only if $\Phi^*(\bar{\theta}{}^\mu)=\theta^\mu$ for all $\mu=1,2,3,4,5$.
  \end{theorem}

This theorem, with the reasoning preceding it, assures that the only possible homogeneous PDE five variables para-CR structures with $I^3\neq 0$ and $I^3{}_{|5}\neq 0$ are those that correspond to the forms $(\theta^1,\theta^2,\dots,\theta^5)$ satisfying the system \eqref{homo1} with \emph{all} coefficients $c_1,c_2,\dots,c_{15}$ being \emph{constants}. If such structures exist, these constants must satisfy the system \eqref{homo1} and its differential consequences $\der(\der(\theta^\mu))=0$ for all $\mu=1,2,3,4,5$. This is a very strong condition which \emph{have only one solution}, given by:
$$\begin{aligned}&c_1=c_2=c_{10}=c_{12}=c_{13}=c_{15}=0,\,\, c_3=6,\,\,c_4=\tfrac32,\,\,c_5=\tfrac12,\,\,c_6=\tfrac18,\,\,c_7=\tfrac{1}{32},\\&c_8=-\tfrac{1}{16},\,\,c_9=-\tfrac12,\,\,c_{11}=\tfrac{3}{16},\,\,c_{14}=-\tfrac18.\end{aligned}$$

It is easy to see that the case when $I^3\neq 0$ and $$I^3{}_{|5}=0$$ in an open set is \emph{impossible}. Indeed, using the EDS from Proposition \ref{susend} and the condition that $I^3{}_{|5}=0$ we see that
$$0=\der(\der(\theta^5))\dz\om^1\dz\om^2=2I^3\om^1\dz\om^2\dz\om^3\dz\om^4\dz\om^5,$$
i.e. that in particular $I^3=0$ in the open set, 
which is a \emph{contradiction}.

Summarizing we have the following 
\begin{corollary}\label{co1}
The only possible two homogeneous models of a PDE five variables para-CR structure with $I^3\neq 0$ must be described by invariant forms $(\theta^1,\theta^2,\theta^3,\theta^4,\theta^5)$ satisfying
 \be
  \begin{aligned}
  \der\theta^1=&\,\epsilon\Big(-6\theta^1\dz\theta^3+\tfrac12\theta^1\dz\theta^4-\tfrac32\theta^1\dz\theta^5\Big)+\theta^2\dz\theta^4\\
  \der\theta^2=&\,\epsilon\Big(-\tfrac{1}{16}\theta^1\dz\theta^2-2\theta^2\dz\theta^3+\tfrac12\theta^2\dz\theta^4-\theta^2\dz\theta^5\Big)-\theta^1\dz\theta^3+\\&\tfrac{1}{32}\theta^1\dz\theta^4-\tfrac18\theta^1\dz\theta^5+\theta^3\dz\theta^4\\
  \der\theta^3=&\,\epsilon\Big(-\tfrac{3}{16}\theta^1\dz\theta^3+\tfrac12\theta^3\dz\theta^4-\tfrac12\theta^3\dz\theta^5\Big)+\tfrac{1}{32}\theta^2\dz\theta^4-\tfrac18\theta^2\dz\theta^5\\
  \der\theta^4=&\,\epsilon\Big(-\tfrac18\theta^1\dz\theta^4+\tfrac14\theta^1\dz\theta^5+4\theta^3\dz\theta^4-\tfrac12\theta^4\dz\theta^5\Big)-\theta^2\dz\theta^5\\
  \der\theta^5=&\,\epsilon\Big(-\tfrac{1}{16}\theta^1\dz\theta^5+2\theta^3\dz\theta^5-\tfrac14\theta^4\dz\theta^5\Big).
  \end{aligned}\label{homo1expl}
  \ee
  Here $\epsilon=\pm1$, and the structures with different values of $\epsilon$ are para-CR nonequivalent.\qed
  \end{corollary}
We will realize these two homogeneous structures with a 5-dimensional symmetry algebra in Section \ref{sechomo}.
\subsection{The case $I^3= 0$ and corresponding homogeneous models}
If $$I^3=0$$ in an open set then Proposition \ref{susend} reduces to:
\begin{proposition}\label{susend1}
  The defining coframe $(\om^1,\om^2,\om^3,\om^4,\om^5)$ of a PDE five variables para-CR structure with $I^3=0$ can be chosen in such a way that it satisfies the following EDS:
  \be
\begin{aligned}
\der \om^1=&-\om^1\dz\varpi_1+\om^2\dz\om^4\\
\der \om^2=&-\om^1\dz\varpi_3+\om^2\dz(\varpi_2-\tfrac12\varpi_1)+\om^3\dz\om^4\\
\der \om^3=&-\om^2\dz\varpi_3+2\om^3\dz\varpi_2+I^1\,\om^1\dz\om^4\\
\der \om^4=&-\om^1\dz\varpi_4-\om^4\dz(\varpi_2+\tfrac12\varpi_1)-\om^2\dz\om^5\\
\der \om^5=&\,\om^4\dz\varpi_4-2\om^5\dz\varpi_2+I^2\,\om^1\dz\om^2.
\end{aligned}\label{sysendu}
\ee
$$\begin{aligned}
\der I^1=&\,I^1{}_{|1}\om^1+I^1{}_{|2}\om^2+I^1{}_{|3}\om^3+I^1{}_{|4}\om^4-\tfrac32 I^1\varpi_1-3I^1\varpi_2\\
\der I^2=&\,I^2{}_{|1}\om^1+I^2{}_{|2}\om^2+I^2{}_{|4}\om^4+I^2{}_{|5}\om^5-\tfrac32 I^2\varpi_1+3I^2\varpi_2\end{aligned}$$
 \be
\begin{aligned}
\der \varpi_1=&\,\om^1\dz\varpi_5+\om^2\dz\varpi_4-\om^4\dz\varpi_3\\
\der \varpi_2=&-\tfrac12\om^2\dz\varpi_4-\tfrac12\om^4\dz\varpi_3-\tfrac{1}{2}I^2{}_{|5}\om^1\dz\om^2+\tfrac{1}{2}I^1{}_{|3}\om^1\dz\om^4-\om^3\dz\om^5\\
\der \varpi_3=&\,\varpi_3\dz(\tfrac12\varpi_1+\varpi_2)+\tfrac12\om^2\dz\varpi_5+\om^3\dz\varpi_4+J^1\om^1\dz\om^2+I^2{}_{|5}\om^1\dz\om^3-\\&I^1{}_{|2}\om^1\dz\om^4+I^1\om^1\dz\om^5-\tfrac{1}{2}I^1{}_{|3}\om^2\dz\om^4\\
\der \varpi_4=&\,\varpi_4\dz(\tfrac12\varpi_1-\varpi_2)+\tfrac12\om^4\dz\varpi_5+\om^5\dz\varpi_3+I^2{}_{|4}\om^1\dz\om^2-I^2\om^1\dz\om^3+\\&\tfrac{1}{2}I^2{}_{|5}\om^2\dz\om^4-I^1{}_{|3}\om^1\dz\om^5\\
\der \varpi_5=&\,\varpi_5\dz\varpi_1+2\varpi_4\dz\varpi_3+J^2\om^1\dz\varpi_3+J^3\om^1\dz\varpi_4+J^4\om^1\dz\om^2+J^5\om^1\dz\om^3+\\&J^6 \om^1\dz\om^4+J^7\om^1\dz\om^5-I^2\om^2\dz\om^3+J^1\om^2\dz\om^4-I^1{}_{|3}\om^2\dz\om^5+\\&I^2{}_{|5}\om^3\dz\om^4-I^1\om^4\dz\om^5.
\end{aligned}\label{sysendu1}
\ee
$$\begin{aligned}\der I^2{}_{|2}=&\,\tfrac{1}{2}(2I^2{}_{|12}+I^2I^2{}_{|5})\om^1+I^2{}_{|22}\om^2+I^2{}_{|24}\om^4+I^2{}_{|25}\om^5-
2I^2{}_{|2}\varpi_1+4I^2{}_{|2}\varpi_2-3I^2\varpi_4,\\
\der I^2{}_{|4}=&\,\tfrac12(2I^2{}_{|14}-3I^2I^1{}_{|3})\om^1+(I^2{}_{|24}-I^2{}_{|1})\om^2-I^2{}_{|2}\om^3+I^2{}_{|44}\om^4+I^2{}_{|45}\om^5-\\&2I^2{}_{|4}\varpi_1+2I^2{}_{|4}\varpi_2+I^2{}_{|5}\varpi_4,\\
\der I^2{}_{|5}=&\,I^2{}_{|15}\om^1+(I^2{}_{|25}+I^2{}_{|4})\om^2+3I^2\om^3+
I^2{}_{|45}\om^4+I^2{}_{|55}\om^5-
\tfrac32I^2{}_{|5}\varpi_1+I^2{}_{|5}\varpi_2,\\
\der I^2{}_{|25}=&\,\tfrac14\big(4I^2{}_{|152}-4I^2{}_{|14}+2(I^2{}_{|5})^2-I^2I^2{}_{|55}\big)\om^1+I^2{}_{|252}\om^2+4I^2{}_{|2}\om^3+I^2{}_{|245}\om^4+\\&I^2{}_{|255}\om^5-2I^2{}_{|25}\varpi_1+2I^2{}_{|25}\varpi_2-3I^2\varpi_3-3I^2{}_{|5}\varpi_4.
\end{aligned}
$$

The coefficients $J^1,J^2,\dots,J^7$ are not important here.
  \end{proposition}
Similarly, Corollary \ref{nurhuj} now takes the form:
\begin{corollary}
  The torsion normalization equations \eqref{norme} applied to the forms \eqref{mogi} with $(\om^1,\om^2,\om^3,\om^4,\om^5)$ with $I^3=0$ as in \eqref{sysend} yield the following para-CR invariant differential system
 \be
\begin{aligned}
\der \theta^1=&-\theta^1\dz\Om_1+\theta^2\dz\theta^4\\
\der \theta^2=&\theta^2\dz(\Om_2-\tfrac12\Om_1)-\theta^1\dz\Om_3+\theta^3\dz\theta^4\\
\der \theta^3=&2\theta^3\dz\Om_2-\theta^2\dz\Om_3+\tfrac{{\rm e}^{3\phi}}{\rho^3}I^1\,\theta^1\dz\theta^4\\
\der \theta^4=&-\theta^2\dz\theta^5-\theta^4\dz(\tfrac12\Om_1+\Om_2)-\theta^1\dz\Om_4\\
\der \theta^5=&-2\theta^5\dz\Om_2+\theta^4\dz\Om_4+\tfrac{{\rm e}^{-3\phi}}{\rho^3}I^2\,\theta^1\dz\theta^2.
\end{aligned}
\label{nurhuju}\ee
\end{corollary}
So if $I^3=0$ in an open set, the structural functions $I^1$ and $I^2$ are \emph{relative invariants} of the considered para-CR structures.

We first analyze the case when
$$I^2\neq 0.$$ 

If the relative invariant $I^2\neq 0$ we can normalize the term at $\theta^1\dz\theta^2$ in $\der\theta^5$ to 1,
$$\frac{\mathrm{e}^{-3\phi}I^2}{\rho^3}=1,$$
reducing the system \eqref{nurhuju} by one dimension, due to the choice of the section
$$\rho=\mathrm{e}^{-\phi}(I^2)^{\tfrac13}.$$
After this normalization the form $\Om_2$ becomes dependent on the forms $\Om_1$, $\theta^\mu$, and thus it disappears from the equations \eqref{nurhuju}. Actually, the entire combination $\Om_2-\tfrac12\Om_1$ disappears from these equations. This in particular gives
$$
  (\der \theta^2)\dz \theta^3\dz\theta^4\dz\theta^5=\Big(-\theta^1\dz\Om_3+\frac{K}{6(I^2)^{\tfrac73}}\theta^1\dz\theta^2\Big)\dz\theta^3\dz\theta^4\dz\theta^5,
$$
with the coefficient $K=3(I^2)^{\tfrac73}u_1+L$. Here $u_1$ is an auxiliary variable introduced when normalizing the system \eqref{nurhuju}. It is analogous to $u_1$ introduced in \eqref{u1}. The quantity $L$ depends on the structural function $I^2$, its derivatives, and the free fiber coordinates $\phi$, $f^2$ and $\bar{f}{}^2$, only. The explicit linear $u_1$-dependence of $K$, where the $u_1$ term is multiplied by $(I^2)^{\tfrac73}$, which is assumed not to vanish, enables us to normalize the coefficient at $\theta^1\dz\theta^2$ in $\der\theta^2$ to 0,
$$K=0.$$
This eliminates the auxiliary variable $u_1$ from the system.

After this normalization we get in particular that:
$$\begin{aligned}\der \theta^4=&-\theta^2\dz\theta^5-\theta^4\dz\Om_1-\theta^1\dz\Om_4+\\&\frac{3\mathrm{e}^{2\phi}\bar{f}{}^2(I^2)^{\tfrac23}-I^2{}_{|2}}{3(I^2)^{\tfrac43}}\theta^2\dz\theta^4+
\frac{\mathrm{e}^{2\phi}I^2{}_{|5}}{3I^2}\theta^4\dz\theta^5.
\end{aligned}$$
This enables for further reduction, by forcing the coefficient of $\theta^2\dz\theta^4$ to vanish. This results in the restriction of the system \eqref{nurhuju} to the section
$$\bar{f}^2=\frac{I^2{}_{|2}}{3\mathrm{e}^{2\phi}(I^2)^{\tfrac23}},$$
on which the form $\Om_4$ becomes dependent on $\theta^\mu$s. Thus, it is not present in the reduced system in which, in particular the differential of $\theta^2$ reads:
$$\begin{aligned}\der \theta^2=&\,\theta^3\dz\theta^4-\theta^1\dz\Om_3+\mathrm{e}^{2\phi}\big(\dots\big)\theta^2\dz\theta^4-\frac{\mathrm{e}^{2\phi}I^2{}_{|5}}{3I^2}\theta^2\dz\theta^5.
\end{aligned}$$
This shows that when $I^3=0$ and $I^2\neq 0$, which we assume in this section, the structural function $I^2{}_{|5}$ is a \emph{relative invariant} of such para-CR structures. In particular, if we have two para-CR structures, one with $I^2{}_{|5}\neq 0$ and the other with $I^2{}_{|5}=0$, then they are locally para-CR nonequivalent.

Let us first concentrate on the case when
$$I^2{}_{|5}\neq 0.$$
In this case we can normalize the term at $\theta^2\dz\theta^5$ in $\der\theta^2$ to be equal to $-\epsilon$, where $$\epsilon=\sgn(\tfrac{I^2{}_{|5}}{I^2}).$$
This results in further reduction of the system \eqref{nurhuju} to the section on which
$$\phi=\tfrac12\log(\tfrac{3\epsilon I^2 }{I^2{}_{|5}}).$$
This makes $\Om_1$ dependent on $\theta^\mu$s only, and eliminates $\Om_1$ from the variables in the reduced EDS. Then in particular, the differential of $\theta^1$ satisfies
$$
  (\der \theta^1)\dz \theta^3\dz\theta^4\dz\theta^5=\frac{9f^2(I^2)^{\tfrac53}-3I^2(I^2{}_{|25}+I^2{}_{|4})+2I^2{}_{|2}I^2{}_{|5}}{3(I^2)^{\tfrac43}I^2{}_{|5}}\theta^1\dz\theta^2\dz\theta^3\dz\theta^4\dz\theta^5.
$$
This enables for the ultimate normalization which kills the $\theta^1\dz\theta^2$ term in $\der\theta^1$. It is obtained by taking the section
$$f^2=\frac{3I^2(I^2{}_{|25}+I^2{}_{|4})-2I^2{}_{|2}I^2{}_{|5}}{9(I^2)^{\tfrac53}}.$$
After this normalization the system \eqref{norhuj} reduces to the original five manifold $M$ on which the para-CR structure is defined. It brings the initial forms $\omega^\mu$ satisfying the system \eqref{sysend} to the fully para-CR invariant forms $\theta^\mu$ on $M$, via the formula $\theta^\mu=S^\mu{}_{\nu}\om^\nu$, where the matrix $S=(S^\mu{}_{\nu})$ is given by:
$$S={\tiny \bma
\frac{\epsilon I^2{}_{|5}}{3(I^2)^{\tfrac13}}&0&0&0&0\\
\frac{3I^2(I^2{}_{|25}+I^2{}_{|4})-2I^2{}_{|2}I^2{}_{|5}}{9(I^2)^{\tfrac53}}&(I^2)^{\tfrac13}&0&0&0\\
\frac{\epsilon(3I^2(I^2{}_{|25}+I^2{}_{|4})-2I^2{}_{|2}I^2{}_{|5})^2}{54(I^2)^{3}I^2{}_{|5}}&\frac{\epsilon(3I^2(I^2{}_{|25}+I^2{}_{|4})-2I^2{}_{|2}I^2{}_{|5})}{3I^2I^2{}_{|5}}&\tfrac{3\eps I^2}{I^2{}_{|5}}&0&0\\
\tfrac{\epsilon I^2{}_{|2}I^2{}_{|5}}{9 (I^2)^{\tfrac53}}&0&0&\tfrac{\epsilon I^2{}_{|5}}{3(I^2)^{\tfrac23}}&0\\
-\tfrac{\epsilon (I^2{}_{|2})^2I^2{}_{|5}}{54 (I^2)^3}&0&0&-\tfrac{\epsilon I^2{}_{|2}I^2{}_{|5}}{9(I^2)^2}&\tfrac{\epsilon I^2{}_{|5}}{3I^2}
\ema}.$$
The resulting para-CR invariant EDS on $M$ is presented in the following statement.
\begin{theorem}
  Every PDE five variables para-CR structure on a 5-dimensional manifold $M$ with the relative invariants $I^3=0$, $I^2\neq 0$ and $I^2{}_{|5}\neq 0$ uniquely defines five 1-forms $(\theta^1,\theta^2,\theta^3,\theta^4,\theta^5)$ on $M$ which satisfy the following exterior differential system
  \be
  \begin{aligned}
  \der\theta^1=&-\epsilon\Big(\theta^1\dz\theta^3+c_4\theta^1\dz\theta^4-c_2\theta^1\dz\theta^5\Big)+\theta^2\dz\theta^4\\
  \der\theta^2=&-\epsilon\Big(c_8\theta^1\dz\theta^2+c_3\theta^2\dz\theta^4+\theta^2\dz\theta^5\Big)+c_7\theta^1\dz\theta^4+c_6\theta^1\dz\theta^5+\theta^3\dz\theta^4\\
  \der\theta^3=&\,\epsilon\Big(c_{10}\theta^1\dz\theta^3+c_1\theta^1\dz\theta^4+(c_4-2c_3)\theta^3\dz\theta^4+(c_2-2)\theta^3\dz\theta^5\Big)-\\&c_{11}\theta^1\dz\theta^2+c_7\theta^2\dz\theta^4+c_6\theta^2\dz\theta^5\\
  \der\theta^4=&\,\epsilon\Big(-c_{9}\theta^1\dz\theta^4+c_{3}\theta^1\dz\theta^5+\theta^3\dz\theta^4+(1-c_2)\theta^4\dz\theta^5\Big)+c_{5}\theta^1\dz\theta^2-\theta^2\dz\theta^5\\
  \der\theta^5=&\,\epsilon\Big(-c_{12}\theta^1\dz\theta^4-c_{10}\theta^1\dz\theta^5+\theta^3\dz\theta^5+(c_4-3c_3)\theta^4\dz\theta^5\Big)+\\&\theta^1\dz\theta^2+c_5\theta^2\dz\theta^4.
  \end{aligned}\label{homo2}
  \ee
  Here, the coefficients $c_1,c_2,\dots,c_{12}$ are totally expressed in terms of the structural functions $I^1,I^2$ and their derivatives. They are uniquely defined by the para-CR structure on $M$ (but their explicit forms are not relevant here) .

  Two different PDE five variables para-CR structures with their corresponding relative invariants $I^3=0$, $I^2\neq 0$ and $I^2{}_{|5}\neq 0$ are locally para-CR equivalent if and only if their corresponding 1-forms $\theta^\mu$ and $\bar{\theta}{}^\mu$  can be transformed to each other by a local diffeomorphism $\Phi$, i.e. if and only if $\Phi^*(\bar{\theta}{}^\mu)=\theta^\mu$ for all $\mu=1,2,3,4,5$.
  \end{theorem}

The above theorem, and the reasoning preceding it, assure that the only possible \emph{homogeneous} PDE five variables para-CR structures with $I^3= 0$, $I^2\neq 0$ and $I^2{}_{|5}\neq 0$ are those that correspond to the forms $(\theta^1,\theta^2,\dots,\theta^5)$ satisfying the system \eqref{homo2} with \emph{all} coefficients $c_1,c_2,\dots,c_{12}$ being \emph{constants}. If such structures exist, these constants must satisfy the system \eqref{homo2} and its differential consequences $\der(\der(\theta^\mu))=0$ for all $\mu=1,2,3,4,5$. This is a very strong condition which \emph{has a one parameter family of solutions} only. This family is given by:
$$\begin{aligned}&c_3=c_4=c_{6}=c_{10}=0,\quad c_1=c_2=c_{11}=c_{12}=1,\quad c_5=c_9=-c_7=-c_8=s,\end{aligned}$$
where $s$ is a real parameter.

It is easy to see that the case when $I^3=0$ and $$I^2\neq 0\quad\mathrm{while}\quad I^2{}_{|5}=0$$ in an open set is \emph{impossible}. Indeed, using the EDS from Proposition \ref{susend1} and the condition that $I^2{}_{|5}=0$ we see that
$$0=\der(\der(I^2))\dz\om^1\dz\om^2\dz\om^4=3I^2\om^1\dz\om^2\dz\om^3\om^4\dz\om^5,$$
i.e. that in particular $I^2=0$ in the neighbourhood, which is a \emph{contradiction}.

The last case to consider when $I^3=0$ is to assume that
$$I^2=0$$
in an open set.

If this is the case we also have $I^2{}_{|5}=0$ in the equations of Proposition \ref{susend1}. Then using the second equation \eqref{sysendu1} we find that
$$0=\der(\der(\varpi_2))\dz\om^4=-I^1{}_{|3}\om^1\dz\om^2\dz\om^4\dz\om^5,$$
which implies that $I^1{}_{|3}=0$ in the open set. Having established this we get that
$$0=\der(\der(\varpi_2))\dz\om^2=-\tfrac32 I^1\om^1\dz\om^2\dz\om^4\dz\om^5,$$
i.e. that also $I^1=0$ in the open set.

Thus, the assumption that $I^3=I^2=0$ in an open set implies that also $I^1=0$ in the same open set, i.e. that \emph{all the fundamental invariants} of the para-CR structure in question \emph{vanish}. Therefore \emph{if}
$$I^3=I^2=0$$ \emph{the corresponding para-CR structure is locally para-CR equivalent to the flat model} described in Section \ref{flatm}. 

Summarizing we have the following
\begin{corollary}\label{co2}
  In the case when $I^3=0$ in an open set the only \emph{homogeneous}  models of a PDE five variables para-CR structure with $I^3=0$ are given either by\\
  \indent \emph{(i)} the flat model represented by the coframe \eqref{flatforms},\\
  or\\
  \indent \emph{(ii)}  must be described by the invariant forms $(\theta^1,\theta^2,\theta^3,\theta^4,\theta^5)$ satisfying
 \be
  \begin{aligned}
  \der\theta^1=&-\epsilon\Big(\theta^1\dz\theta^3+\theta^1\dz\theta^5\Big)+\theta^2\dz\theta^4\\
  \der\theta^2=&\,\epsilon\Big(s\theta^1\dz\theta^2-\theta^2\dz\theta^5\Big)-s\theta^1\dz\theta^4+\theta^3\dz\theta^4\\
  \der\theta^3=&\,\epsilon\Big(\theta^1\dz\theta^4-\theta^3\dz\theta^5\Big)-\theta^1\dz\theta^2-s\theta^2\dz\theta^4\\
  \der\theta^4=&\,\epsilon\Big(-s\theta^1\dz\theta^4+\theta^3\dz\theta^4\Big)+s\theta^1\dz\theta^2-\theta^2\dz\theta^5\\
  \der\theta^5=&\,\epsilon\Big(-\theta^1\dz\theta^4+\theta^3\dz\theta^5\Big)+\theta^1\dz\theta^2+s\theta^2\dz\theta^4.
  \end{aligned}\label{homo2expl}
  \ee
  Here $\epsilon=\pm1$, $s$ is an arbitrary real number, and the structures with different values of $(\epsilon,s)$ are para-CR nonequivalent.
  \end{corollary}
We will realize these homogeneous structures in the next Section.
\subsection{All homogeneous models}\label{sechomo}
Combining Corollaries \ref{co1} and \ref{co2} we have the following proposition.
\begin{proposition}\label{classs}
   The only \emph{homogeneous}  models of a PDE five variables para-CR structure are given either by\\
  \indent \emph{(i)} the flat model represented by the coframe \eqref{flatforms},\\
  or\\
  \indent \emph{(ii)}  must be described by the invariant forms $(\theta^1,\theta^2,\theta^3,\theta^4,\theta^5)$ satisfying the system \eqref{homo1expl}\\
  or\\
  \indent \emph{(iii)} must be described by the invariant forms $(\theta^1,\theta^2,\theta^3,\theta^4,\theta^5)$ satisfying the system \eqref{homo2expl}.

The structure \emph{(i)} has $\sog(3,2)$ as the local group of para-CR symmetries, whereas the structures described in \emph{(ii)} and \emph{(iii)} have maximal local group of para-CR symmetries of dimension 5. There are no homogeneous models of such para-CR structures with a local symmetry group of exact dimensions six to nine.  
\end{proposition}
In this section we will show that all the structures described in this proposition do exist, and will associate a system of PDEs on the plane, of the form \eqref{sysf}-\eqref{sysfc}, to each of them.

We have the following statement.
\begin{theorem}
\label{theorem-all-homogeneous-models}
  All homogeneous models of PDE five variables para-CR structures are given, in terms of their defining PDEs, by\\
  \indent \emph{(i)}\hspace{0.1cm} $ z_y=\tfrac14 (z_x)^2\quad \&\quad z_{xxx}=0;$\hspace{2.25cm} it is the flat case,\\
  \indent \emph{(ii)} $ z_y=\tfrac14 (z_x)^2\quad \& \quad z_{xxx}=(z_{xx})^3;$\hspace{1.45cm} this is the case corresponding to\\ \indent \hspace{7.2cm}\emph{(ii)} in 
  Proposition \ref{classs} with \\ \indent \hspace{7.1cm} $\epsilon=-\sgn(z_{xx})$ ,\\
  \indent \emph{(iiia)} $ z_y=\tfrac14 (z_x)^b\,\, \& \,\,z_{xxx}=(2-b)\frac{(z_{xx})^2}{z_x};$\hspace{0.6cm} this is the case corresponding to \\
  \indent \hspace{7.1cm} \emph{(iii)} in  Proposition \ref{classs} with \\
  \indent \hspace{7.1cm} $\epsilon=-\sgn(z_{xx})$ and $s\leq -3\cdot 2^{-\tfrac53}$;\\
  \indent \hspace{7.1cm} of course $z_x>0$,\\
  \indent \emph{(iiib)} $ z_y=f(z_x)\quad \& \quad z_{xxx}=h(z_x)\big(z_{xx}\big)^2$, 
  where the function $f$ is determined\\
  \indent by the implicit equation
  \be\big(z_x^2+f(z_x)^2\big)\mathrm{exp}\Big(2b\,\mathrm{arctan}\tfrac{bz_x-f(z_x)}{z_x+bf(z_x)}\Big)=1+b^2\label{ef}\ee \indent and \be h(z_x)=\frac{(b^2-3)z_x-4bf(z_x)}{\big(f(z_x)-bz_x\big)^2};\label{eh}\ee
  \indent this is the case corresponding to \emph{(iii)} with $\epsilon=-\sgn(z_{xx})$ and $s>-3\cdot 2^{-\tfrac53}$. 
  \end{theorem}
\begin{proof}
  We first show that the case (ii) in the theorem realizes the EDS \eqref{homo1expl}, and therefore exhausts all possible homogeneous PDE five variables para-CR structures with $I^3\neq 0$.

  We start with the PDE system
  \be z_y=\tfrac14(z_x)^2\,\,\&\,\,z_{xxx}=(z_{xx})^3\label{sy1}\ee and associate with it the coframe $(\om^\mu)$, $\mu=1,2,\dots,5$ via \eqref{iniom}. Explicitly, we have:
  \be\begin{aligned}
    \om^1=&\der z-p\der x-\tfrac14p^2\der y\\
    \om^2=&\der p-r\der x-\tfrac12p r \der y\\
    \om^3=&\der r-r^3\der x-\tfrac12r^2(1+pr)\der y\\
    \om^4=&\der x\\
    \om^5=&\der y.
      \end{aligned}
  \label{exa11}\ee
  Now, it is easy to check that the matrix
  $$S=\bma
  8\epsilon r^3&0&0&0&0\\
  2r^3&2r&0&0&0\\
  \tfrac{\epsilon r^3}{4}&0&\tfrac{\epsilon}{2r}&0&0\\
  0&0&0&-4\epsilon r^2&-2\epsilon p r^2\\
  0&0&0&0&\epsilon r
  \ema,$$
  which has values in the allowed para-CR group $G_0$ as in \eqref{grG0}, transforms the coframe \eqref{exa11} defining the PDE 
five variables
para-CR structure of the system \eqref{sy1} to the invariant forms $\theta^\mu=S^\mu{}_{\nu}\om^\nu$ satisfying the homogeneous system \eqref{homo1expl}, provided that $\eps=-\sgn(r)=-\sgn(z_{xx})$. 

  Now we show that the EDS \eqref{homo2expl} with $s \leq -3\cdot 2^{-\tfrac53}$ is realized by the PDE five variables para-CR structure associated with the system
  \be
  z_y=\tfrac14(z_x)^b\,\,\&\,\,z_{xxx}=(2-b)\frac{(z_{xx})^2}{z_x},\label{sy2}
  \ee
  with $b=\mathrm{const}$.

Using this PDE system we define the corresponding PDE
five variables para-CR structure via the coframe \eqref{iniom}. This reads:
  \be\begin{aligned}
    \om^1=&\,\der z-p\der x-\tfrac14p^b\der y\\
    \om^2=&\,\der p-r\der x-\tfrac14bp^{b-1} r \der y\\
    \om^3=&\,\der r-\frac{(2-b)r^2}{p}\der x-\tfrac14br^2p^{b-2}\der y\\
    \om^4=&\,\der x\\
    \om^5=&\,\der y.
      \end{aligned}
  \label{exa12}\ee
  Now introducing the matrix
$$S=\bma
  \frac{\epsilon rt^2}{9p^2}&0&0&0&0\\
  \frac{-(b+1) rt}{9p^2}&\frac{t}{3p}&0&0&0\\
  \tfrac{-\epsilon r(1-7b+b^2)}{18p^2}&\tfrac{\epsilon(b-5)}{3p}&\tfrac{\epsilon}{r}&0&0\\
  \frac{-\epsilon(b+1) rt}{9p^2}&0&0&\frac{-\epsilon rt}{3p}&\frac{-\epsilon brp^{b-2}t}{12}\\
  \tfrac{\epsilon r(1-7b+b^2)}{18p^2}&0&0&\tfrac{\epsilon(1-2b)r}{3p}&\tfrac{\epsilon (b-2)brp^{b-2}}{12}
  \ema,$$
 with  $t=\big((b-2)(b+1)(2b-1)\big)^{\tfrac13}$ and $\epsilon=-\sgn(z_{xx})$, which again has values in the allowed para-CR group $G_0$ as in \eqref{grG0}, it is easy to see that the transformed coframe $\theta^\mu=S^\mu{}_{\nu}\om^\nu$ with $\omega^\mu$s as in \eqref{exa12}, satisfies the homogeneous system \eqref{homo2expl} with 
\[
s
\,=\,
-\,\frac{3}{2}\,
\frac{1-b+b^2}{\big[(b-2)(b+1)(2b-1)\big]^{2/3}},
\ \ \ \ \ \ \ \ \ \ \ \ 
\frac{ds}{db}
\,=\,
\frac{27}{2}\,
\frac{b\,(b-1)}{
\big[
(b-2)\,(b+1)\,(2\,b-1)
\big]^{5/3}}.
\]

\includegraphics[scale=0.30]{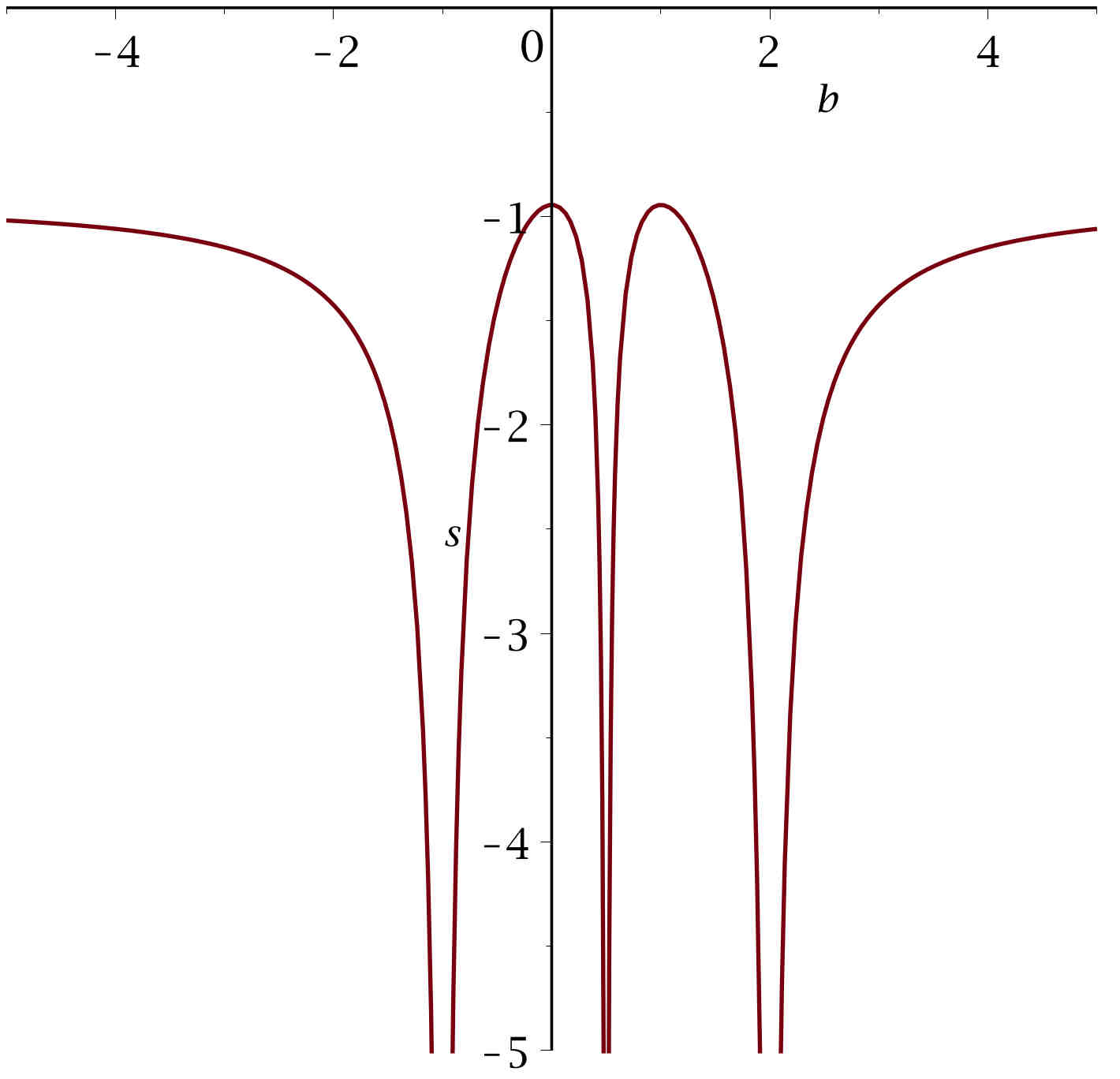}
\ \ \ \ \ \ \ 
\includegraphics[scale=0.30]{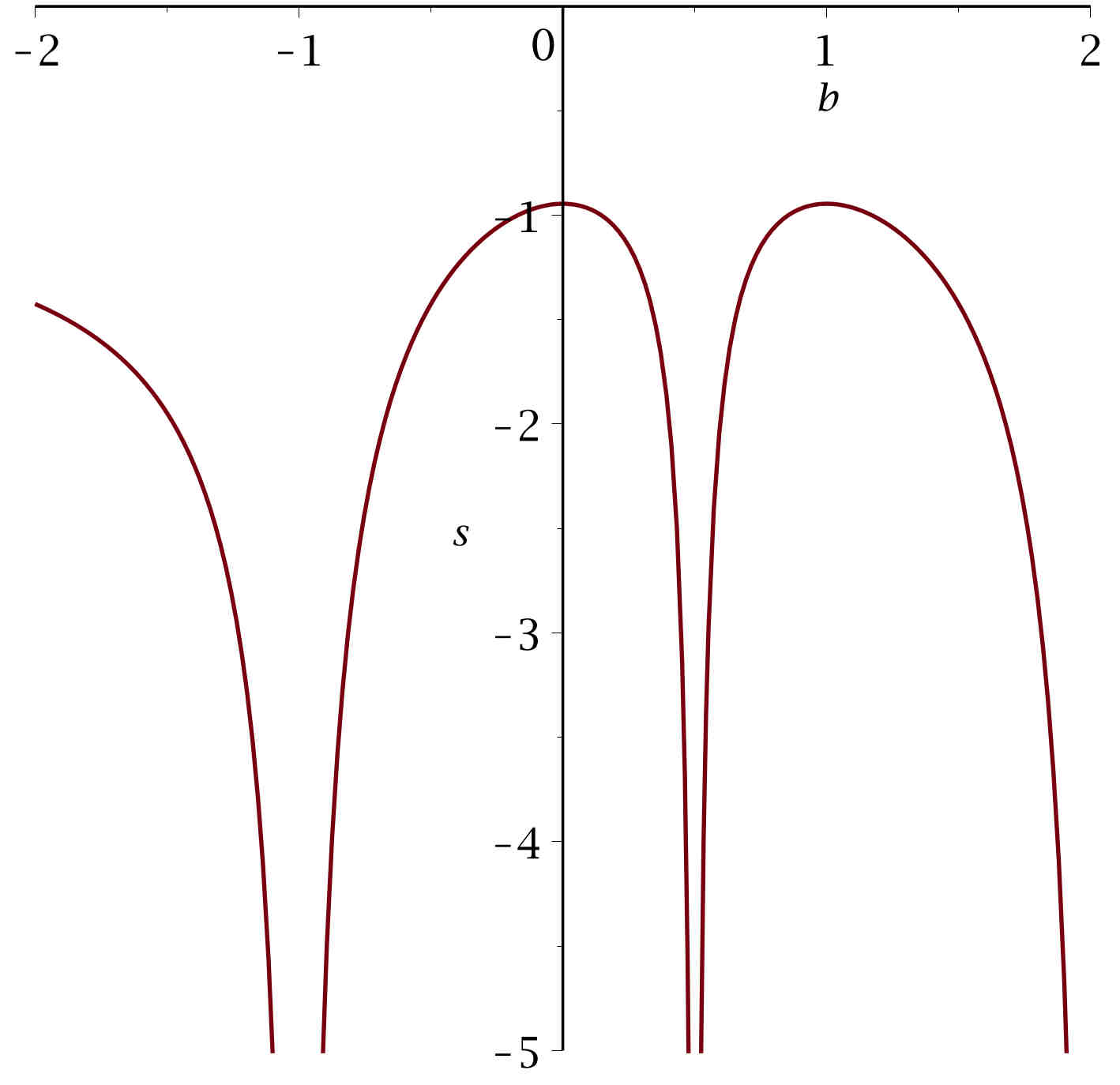}
\ \ \ \ \ \ \ 
\includegraphics[scale=0.30]{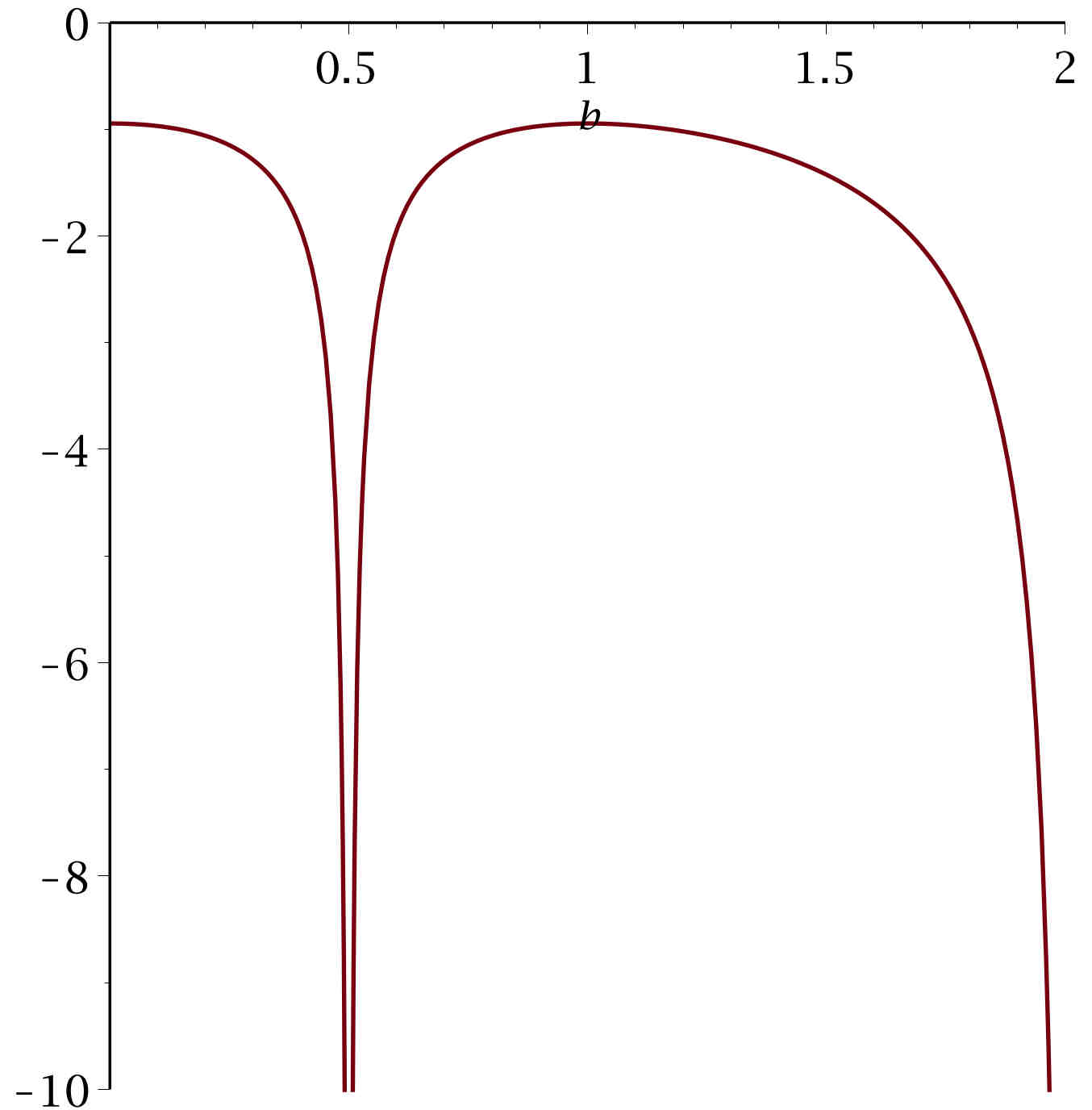}

\medskip\noindent
Looking at the range of the function $b\mapsto s=s(b)$ we see that the PDE system \eqref{sy2} realizes the EDS \eqref{homo2expl} for all the values of $s\leq -3\cdot (2)^{-\tfrac53}$. 
For this it is enough to take $1\leq b<2$.

\smallskip 

The last part of the proof is to show that the PDE system
\be
z_y=f(z_x)\,\,\&\,\,z_{xxx}=h(z_x)(z_{xx})^2,
\label{sy3}
\ee
with $f=f(z_x)$ and $h=h(x)$ as in \eqref{ef}-\eqref{eh} corresponds to a PDE para-CR structure realizing the system \eqref{homo2expl} with the parameter $s$ belonging to the remaining range $s>-3\cdot (2)^{-\tfrac53}$. This is a bit more tricky since we have no explicit dependence of $f$ and $h$ on $z_x$. To deal with this situation we observe that the general solution to the 
equations \eqref{sy3}, \eqref{ef}-\eqref{eh} is
$$z=\exp\Big(b\, \mathrm{arctan}\big(\tfrac{y+\bar{y}}{x+\bar{x}}\big)\Big)\sqrt{(x+\bar{x})^2+(y+\bar{y})^2}-\bar{z},$$
  where $\bar{x},\bar{y},\bar{z}$ are integration constants. It is now convenient to introduce a new variable
\[
u
=
\frac{y+\bar{y}}{x+\bar{x}}
\]
and to write down the para-CR coframe \eqref{iniom} for the para-CR structure associated with this PDE in terms of five variables $(y,\bar{x},\bar{y},\bar{z},u)$ rather than the variables $(x,y,z,z_x,z_{xx})$. This coframe reads:
  \be\begin{aligned}
  \om^1=&-\der\bar{z}+\frac{\exp\big(b\, \mathrm{arctan}\,u\big)}{\sqrt{1+u^2}}\Big((1-bu)\der\bar{x}+(b+u)\der\bar{y}\Big)\\
  \om^2=&\,\frac{\exp\big(b\, \mathrm{arctan}\,u\big)(1+b^2)u^2}{(1+u^2)^{\tfrac32}(y+\bar{y})}\Big(u\der\bar{x}-\der\bar{y}\Big)\\
  \om^3=&\,\frac{-\exp\big(b\, \mathrm{arctan}\,u\big)(1+b^2)u^3}{(1+u^2)^{\tfrac52}(y+\bar{y})^2}\Big((bu+3)u\der\bar{x}+(u^2-bu-2)\der\bar{y}\Big)\\
  \om^4=&\,\der\Big(\tfrac{y+\bar{y}}{u}-\bar{x}\Big)\\
  \om^5=&\,\der y.
  \end{aligned}\ee
  Now it is easy to see that the transformation $\om^\mu\mapsto \theta^\mu=S^\mu{}_\nu\om^\nu$ with
  $$S=\bma\tfrac{\epsilon a t^2 u}{3\sqrt{1+u^2}(y+\bar{y})}&0&0&0&0\\
  \tfrac{2a t b u}{3\sqrt{1+u^2}(y+\bar{y})}&\tfrac{a t\sqrt{1+u^2}}{u}&0&0&0\\
  \tfrac{\epsilon a (9+5b^2)u}{18(1+b^2)\sqrt{1+u^2}(y+\bar{y})}&\tfrac{2\epsilon a(3+2bu)\sqrt{1+u^2}}{3(1+b^2)u^2}&\tfrac{\epsilon a (1+u^2)^{\tfrac32}(y+\bar{y})}{(1+b^2)u^3}&0&0\\
  \tfrac{2\epsilon atb u}{3\sqrt{1+u^2}(y+\bar{y})}&0&0&\tfrac{-\epsilon t(1+b^2)u^2}{(1+u^2)(y+\bar{y})}&\tfrac{\epsilon t(1+b^2)u}{(1+u^2)(y+\bar{y})}\\
  \tfrac{-\epsilon a(9+5b^2)u}{18(1+b^2)\sqrt{1+u^2}(y+\bar{y}}&0&0&\tfrac{\epsilon u(bu-3)}{3(1+u^2)(y+\bar{y})}&\tfrac{-\epsilon u(b+3u)}{3(1+u^2)(y+\bar{y})}
  \ema,$$
  and with $a=\mathrm{e}^{-b\,\mathrm{arctan}\,u}$ and $t=\tfrac{\big(2b(9+b^2)\big)^{\tfrac13}}{3(1+b^2)}$,
  brings this coframe to a para-CR invariant coframe $\theta^\mu$ satisfying the homogeneous system \eqref{homo2expl} with 
\[
s
=
-\,\frac{3}{2}\,
\frac{b^2-3}{\big[2b(9+b^2)\big]^{2/3}},
\ \ \ \ \ \ \ \ \ \  
\frac{ds}{db}
\,=\,
-\,\frac{27}{2^{2/3}}\,
\frac{b^2+1}{\big[b\,(b^2+9)\big]^{5/3}}.
\]

\begin{center}
\includegraphics[scale=0.30]{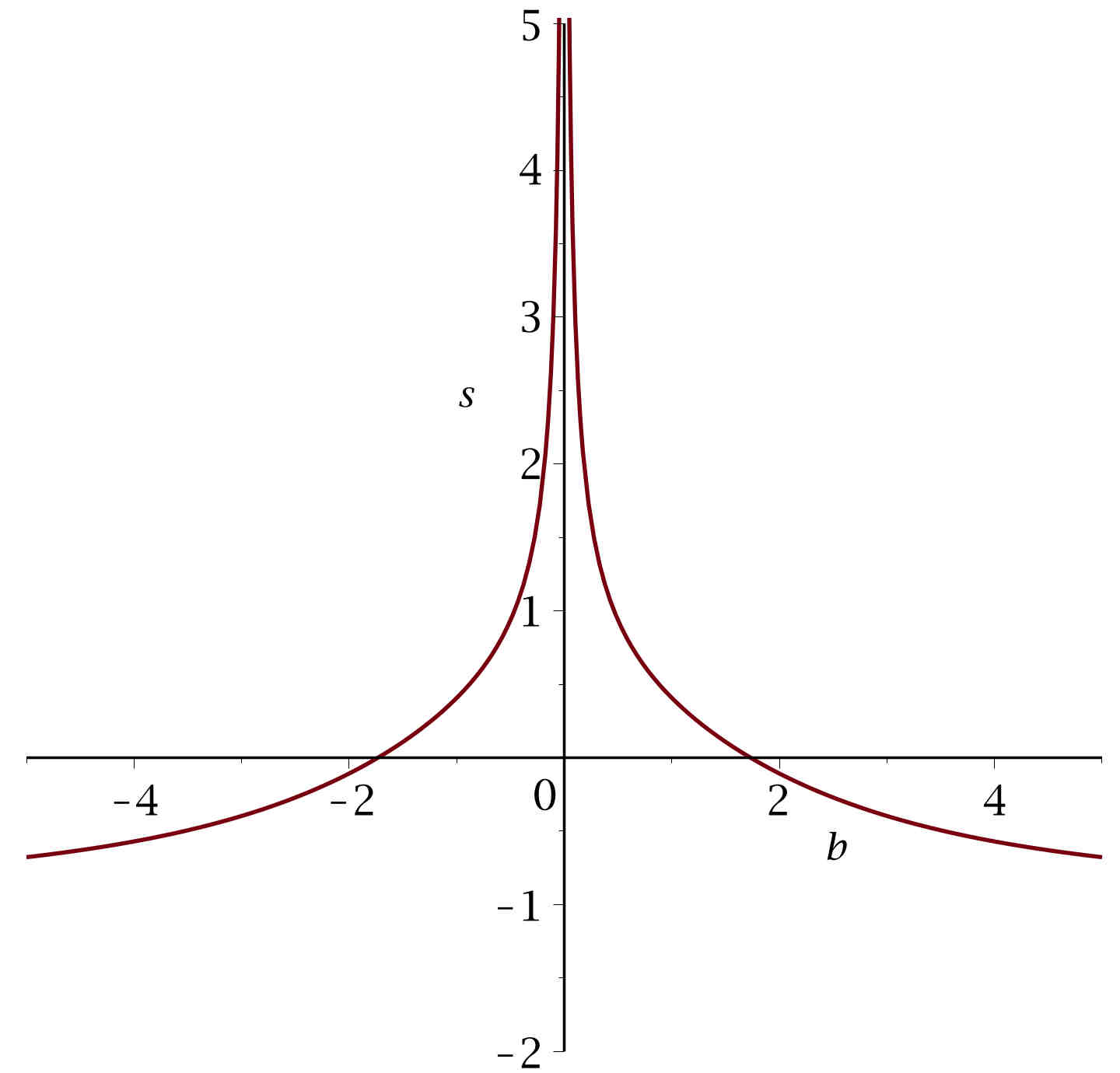}
\end{center}

\medskip\noindent
The range of this even function
$s \longmapsto b(s)$ is $s>-3\cdot (2)^{-\tfrac53}$. To achieve
all the values of $s>-3\cdot (2)^{-\tfrac53}$ it is enough to take
$b>0$.
\end{proof}

\section{Lie Symmetry Algebras}
\label{Lie-symmetry-algebras}

The point automorphism groups for cases (i), (ii), (iiia), (iiib)
can be determined infinitesimally. Indeed, a vector field
with unknown coefficients $A^i = A^i(x,y,z,p,r)$, $i = 1, \dots, 5$:
\[
X
\,:=\,
A^1\,\partial_x
+
A^2\,\partial_y
+
A^3\,\partial_z
+
A^4\,\partial_p
+
A^5\,\partial_r,
\]
should act on $1$-forms as the matrix~(2.4),
so that:
\begin{equation}
\aligned
\label{Lie-X-Omega}
0
&
\,=\,
\mathcal{L}_X(\omega^1)
\dz
\omega^1,
\\
0
&
\,=\,
\mathcal{L}_X(\omega^2)
\wedge
\omega^1\wedge\omega^2\wedge\omega^3,
\\
0
&
\,=\,
\mathcal{L}_X(\omega^3)
\wedge
\omega^1\wedge\omega^2\wedge\omega^3,
\\
0
&
\,=\,
\mathcal{L}_X(\omega^4)
\wedge
\omega^1\wedge\omega^4\wedge\omega^5,
\\
0
&
\,=\,
\mathcal{L}_X(\omega^5)
\wedge
\omega^1\wedge\omega^4\wedge\omega^5.
\endaligned
\end{equation}
For instance, in case (ii), the first equation writes:
\[
\aligned
\mathcal{L}_x(\omega^1)
\wedge
\omega^1
&
\,=\,
\der x\wedge\der y\,
\Big[
p\,
\big(
A_y^3
-
\tfrac{1}{4}\,p^2\,A_y^2
-
p\,A_y^1
-
\tfrac{1}{4}\,p\,A^4
-
\tfrac{1}{4}\,p\,A_x^3
+
\tfrac{1}{16}\,p^3\,A_x^2
+
\tfrac{1}{4}\,A_x^1
\big)
\Big],
\\
&
\ \ 
+
\der x\wedge\der z\,
\Big[
p\,A_z^3
-
\tfrac{1}{4}\,p^3\,A_z^2
-
p^2\,A_z^1
+
A_x^3
-
\tfrac{1}{4}\,p^2\,A_x^2
-
p\,A_x^1
-
A^4
\Big]
\\
&
\ \ 
+
\der x\wedge\der p\,
\Big[
p\,
\big(
A_p^3
-
\tfrac{1}{4}\,p^2\,A_p^2
-
p\,A_p^1
\big)
\Big]
+
\der x\wedge\der r\,
\Big[
p\,
\big(
A_r^3
-
\tfrac{1}{4}\,p^2\,A_r^2
-
p\,A_r^1
\big)
\Big]
\\
&
\ \ 
+
\der y\wedge\der z\,
\Big[
\tfrac{1}{4}\,
p^2\,A_z^3
-
\tfrac{1}{16}\,p^4\,A_z^2
-
\tfrac{1}{4}\,p^3\,A_z^1
+
A_y^3
-
\tfrac{1}{4}\,p^2\,A_y^2
-
p\,A_y^1
-
\tfrac{1}{2}\,A^4
\Big]
\\
&
\ \ 
+
\der y\wedge\der p\,
\Big[
p^2\,
\big(
\tfrac{1}{4}\,A_p^3
-
\tfrac{1}{16}\,p^2\,A_p^2
-
\tfrac{1}{4}\,p\,A_p^1
\big)
\Big]
+
\der y\wedge\der r\,
\Big[
p^2\,
\big(
\tfrac{1}{4}\,A_r^3
-
\tfrac{1}{16}\,p^2\,A_r^2
-
\tfrac{1}{4}\,p\,A_r^1
\big)
\Big]
\\
&
\ \ 
+
\der z\wedge\der p\,
\Big[
-A_p^3
+
\tfrac{1}{4}\,p^2\,A_p^2
+
p\,A_p^1
\Big]
+
\der z\wedge\der r\,
\Big[
-A_r^3
+
\tfrac{1}{4}\,p^2\,A_r^2
+
p\,A_r^1
\Big].
\endaligned
\]
Solving this linear system of partial differential equations,
we get

\begin{corollary}
The Lie algebra of infinitesimal point automorphisms of the flat model
{\bf (i)} is simple, isomorphic to $\mathfrak{so}_{3,2}(\R)$, with the
$10$ generators:
\[
\aligned
X_1
&
\,:=\,
xy\,\partial_x
+
y^2\,\partial_y
-
x^2\,\partial_z
-
(py+2\,x)\,
\partial_p
-
(2\,ry+2)\,
\partial_r,
\\
X_2
&
\,:=\,
-\,(x^2-yz)\,
\partial_x
-
2\,xy\,\partial_y
-
2\,xz\,\partial_z
-
\big(
\tfrac{1}{2}\,p^2y
+
2\,z
\big)\,
\partial_p
-
\big(
pry
-
2\,rx
+
2\,p
\big)\,
\partial_r,
\\
X_3
&
\,:=\,
y\,\partial_x
-
2\,x\,\partial_z
-
2\,\partial_p,
\\
X_4
&
\,:=\,
xz\,\partial_x
-
x^2\,\partial_y
+
z^2\,\partial_z
-
\big(
\tfrac{1}{2}\,
p^2x
-
pz
\big)\,
\partial_p
+
\big(
\tfrac{1}{2}\,
p^2
-
prx
\big)\,
\partial_r,
\\
X_5
&
\,:=\,
z\,\partial_x
-
2\,x\,\partial_y
-
\tfrac{1}{2}\,p^2\,\partial_p
-
pr\,\partial_r,
\\
X_6
&
\,:=\,
x\,\partial_x
+
2\,z\,\partial_z
+
p\,\partial_p,
\\
X_7
&
\,:=\,
\partial_x,
\\
X_8
&
\,:=\,
y\,\partial_y
-
z\,\partial_z
-
p\,\partial_p
-
r\,\partial_r,
\\
X_9
&
\,:=\,
\partial_y,
\\
X_{10}
&
\,:=\,
\partial_z,
\endaligned
\]
having commutator table:
\[
\begin{tabular} [t] { l | c c c c c c c c c c }
& $X_1$ & $X_2$ & $X_3$ & $X_4$ & $X_5$
& $X_6$ & $X_7$ & $X_8$ & $X_9$ & $X_{10}$
\\
\hline
$X_1$ & $0$ & $0$ & $0$ & $0$ & $-X_2$ 
& 
$0$ & $-X_3$ & $-X_1$ & $-X_6-2X_8$ & $0$
\\
$X_2$ & $*$ & $0$ & $2X_1$ & $0$ & $2X_4$ 
& 
$-X_2$ & $2X_6+2X_8$ & $0$ & $-X_5$ & $-X_3$
\\
$X_3$ & $*$ & $*$ & $0$ & $X_2$ & $-2X_8$ 
& 
$X_3$ & $2X_{10}$ & $-X_3$ & $-X_7$ & $0$
\\
$X_4$ & $*$ & $*$ & $*$ & $0$ & $0$ 
& 
$-2X_4$ & $-X_5$ & $X_4$ & $0$ & $-X_6$
\\
$X_5$ & $*$ & $*$ & $*$ & $*$ & $0$ 
& 
$-X_5$ & $2X_9$ & $X_5$ & $0$ & $-X_7$
\\
$X_6$ & $*$ & $*$ & $*$ & $*$ & $*$ 
& 
$0$ & $-X_7$ & $0$ & $0$ & $-2X_{10}$
\\
$X_7$ & $*$ & $*$ & $*$ & $*$ & $*$ 
& 
$*$ & $0$ & $0$ & $0$ & $0$
\\
$X_8$ & $*$ & $*$ & $*$ & $*$ & $*$ 
& 
$*$ & $*$ & $0$ & $-X_9$ & $X_{10}$
\\
$X_9$ & $*$ & $*$ & $*$ & $*$ & $*$ 
& 
$*$ & $*$ & $*$ & $0$ & $0$
\\
$X_{10}$ & $*$ & $*$ & $*$ & $*$ & $*$ 
& 
$*$ & $*$ & $*$ & $*$ & $0$
\end{tabular}
\]
\end{corollary}

In the CR context, 
observe that if $S^2 \subset \R^3$ is an affinely homogeneous
parabolic surface, then the tube $M^5 := S^2 \times i\R^3$
has transitive holomorphic symmetry algebra $\mathfrak{hol}(M)$,
with an {\em Abelian ideal} $\mathfrak{a} := {\rm Span}\,
\{ i\partial_{z_1}, i\partial_{z_2}, i\partial_w\}$.
Conversely, for an $M^5 \in \mathfrak{C}_{2,1}$, 
it is not difficult to show that if $\mathfrak{hol}
(M) \supset \mathfrak{a}$ contains an Abelian ideal $\mathfrak{a}$
with ${\rm rank}_\C \mathfrak{a} = 3$, then $M^5 \cong S^2 \times
i\R^3$ is biholomorphically equivalent to the tube over an affinely
homogeneous parabolic surface $S^2 \subset \R^3$.

In the para-CR context, 
as can be read off from commutator tables,
all the Lie algebras in cases 
{\bf (i)}, {\bf (ii)}, {\bf (iiia)}, {\bf (iiib)}
have a $3$-dimensional abelian ideal.
\begin{corollary}
The Lie algebras of infinitesimal point automorphisms of the
homogeneous models {\bf (ii)}, {\bf (iiia)}, {\bf (iiib)} are all
$5$-dimensional and solvable, and are given in the $(x,y,z,
p,r)$-space by the following generators together with their Lie
brackets:
\begin{footnotesize}
\[
\begin{array}{ll}
\aligned
X_1
& 
:= 
x\,\partial_x 
+ 
\tfrac{1}{2}\,y\,
\partial_y 
+ 
\tfrac{3}{2}\,z\,
\partial_z 
+ 
\tfrac{1}{2}\,p\,
\partial_p 
- 
\tfrac{1}{2}\,r\,
\partial_r,
\\
X_2
&
\,:=\,
y\,\partial_x
-
2\,x\,\partial_z
-
2\,\partial_p,
\\
\text{\bf (ii)}
\ \ \ \ \ \ \ 
X_3
&
\,:=\,
\partial_x,
\\
X_4
&
\,:=\,
\partial_y,
\\
X_5
&
\,:=\,
\partial_z,
\endaligned
\ \ \ \ \ \ \ \ \ \ \ \ \ \ \ 
\aligned
\begin{tabular} [t] { l | c c c c c }
& $X_1$ & $X_2$ & $X_3$ & $X_4$ & $X_5$
\\
\hline
$X_1$ & $0$ & $-\frac{1}{2}X_2$ & $-X_3$ & $-\frac{1}{2}X_4$ & 
$-\frac{3}{2}X_5$
\\
$X_2$ & $*$ & $0$ & $2X_5$ & $-X_3$ & $0$
\\
$X_3$ & $*$ & $*$ & $0$ & $0$ & $0$
\\
$X_4$ & $*$ & $*$ & $*$ & $0$ & $0$
\\
$X_5$ & $*$ & $*$ & $*$ & $*$ & $0$ 
\end{tabular}
\endaligned
\end{array}
\]

\[
\begin{array}{ll}
\aligned
X_1
& 
:= 
x\,\partial_x
+
\frac{b\,z}{b-1}\,
\partial_z
+
\frac{p}{b-1}\,
\partial_p
-
\frac{r(b-2)}{b-1}\,
\partial_r,
\\
X_2
&
\,:=\,
y\,\partial_y
-
\frac{z}{b-1}\,\partial_z
-
\frac{p}{b-1}\,\partial_p
-
\frac{r}{b-1}\,\partial_z,
\\
\text{\bf (iiia)}
\ \ \ \ \ \ \ 
X_3
&
\,:=\,
\partial_x,
\\
X_4
&
\,:=\,
\partial_y,
\\
X_5
&
\,:=\,
\partial_z,
\endaligned
\ \ \ \ \ \ \ \ \ \ \ \ \ \ \ 
\aligned
\begin{tabular} [t] { l | c c c c c }
& $X_1$ & $X_2$ & $X_3$ & $X_4$ & $X_5$
\\
\hline
$X_1$ & $0$ & $0$ & $-X_3$ & $0$ & $-\frac{b}{b-1}X_5$
\\
$X_2$ & $*$ & $0$ & $0$ & $-X_4$ & $\frac{1}{b-1}X_5$
\\
$X_3$ & $*$ & $*$ & $0$ & $0$ & $0$
\\
$X_4$ & $*$ & $*$ & $*$ & $0$ & $0$
\\
$X_5$ & $*$ & $*$ & $*$ & $*$ & $0$ 
\end{tabular}
\endaligned
\end{array}
\]

\[
\begin{array}{ll}
\aligned
X_1
& 
:= 
x\,\partial_x
+
y\,\partial_y
+
z\,\partial_z
-
r\,\partial_r,
\\
X_2
&
\,:=\,
-y\,\partial_x
+
x\,\partial_y
+
\omega\,z\,\partial_z
+
\big(-F+\omega\,p)\,\partial_p
+
(-2DF+\omega\,r)\,\partial_r
\\
\text{\bf (iiib)}
\ \ \ \ \ \ \ 
X_3
&
\,:=\,
\partial_x,
\\
X_4
&
\,:=\,
\partial_y,
\\
X_5
&
\,:=\,
\partial_z,
\endaligned
\ \ \ \ \ \ \
\aligned
\begin{tabular} [t] { l | c c c c c }
& $X_1$ & $X_2$ & $X_3$ & $X_4$ & $X_5$
\\
\hline
$X_1$ & $0$ & $0$ & $-X_3$ & $-X_4$ & $-X_5$
\\
$X_2$ & $*$ & $0$ & $-X_4$ & $X_3$ & $-\omega X_5$
\\
$X_3$ & $*$ & $*$ & $0$ & $0$ & $0$
\\
$X_4$ & $*$ & $*$ & $*$ & $0$ & $0$
\\
$X_5$ & $*$ & $*$ & $*$ & $*$ & $0$ 
\end{tabular}
\endaligned
\end{array}
\]
\end{footnotesize}
\end{corollary}

\section{Link with Homogeneous Parabolic Surfaces}
\label{link-homogeneous-parabolic-surfaces}

We close this article by some comments on the classification
Theorem~{\ref{theorem-all-homogeneous-models}}.
Para-CR structures can be defined either from the point of view of
{\sl submanifolds of solutions}~{\cite{Merker-2008}},
of from the point of view of {\sl exterior differential 
systems}~{\cite{Hill-Nurowski-2010}}. From the first viewpoint, 
after `dividing' by the $3$-dimensional Abelian ideal
present in all the homogeneous models 
{\bf (i)}, {\bf (ii)}, {\bf (iiia)}, {\bf (iiib)},
namely after passing to a quotient, 
one can convince oneself that
one obtains {\em affinely homogeneous surfaces
$S^2 \subset \R^3$}.

The complete classification of $A_3(\R)$-homogeneous surfaces $S^2
\subset \R^3$ was terminated by
Doubrov-Komrakov-Rabinovich~{\cite{Doubrov-Komrakov-Rabinovich-1996}},
and re-done by Eastwood-Ezhov in~{\cite{Eastwood-Ezhov-1999}} who
employed the power series method.  This classification incorporates the
classification of $A_3(\R)$-homogeneous {\em parabolic} surfaces,
namely surfaces locally graphed as $z = F(x,y)$ with
\[
F_{xx}
\,\neq\,0
\,\equiv\,
\left\vert\!
\begin{array}{cc}
F_{xx} & F_{xy}
\\
F_{yx} & F_{yy}
\end{array}
\!\right\vert.
\]
From~{\cite{Fels-Kaup-2008}}, the list is as follows.

\smallskip\noindent{\bf (1)}\,
$\big\{ x_1^2 + x_2^2 = x_3^2, \,\, 
x_3 >0 \big\}$ the future light cone,
having infinitesimal symmetries 
$x_1 \partial_{x_1} + x_2 \partial_{x_2} + x_3 \partial_{x_3}$,
$-x_2\partial_{x_1} + x_1 \partial_{x_2}$;

\smallskip\noindent{\bf (2a)}\,
$\big\{ r ( \cos t, \sin t, e^{ \omega t})\in \mathbb{ R}^3: \, r
\in \mathbb{ R}^+ \ \text{\rm and}\ t \in \mathbb{ R} \big\}$ with
$\omega > 0$ arbitrary, graphed as
$u = \sqrt{x^2+y^2}\,
e^{\omega \arctan \frac{y}{x}}$, 
having symmetries $x\partial_x 
+ y \partial_y + u \partial_u$, $-y\partial_x + 
x\partial_y + \omega u \partial_u$;

\smallskip\noindent{\bf (2b)}\,
$\big\{ r ( 1, t, e^t) \in \mathbb{ R}^3: \, r \in \mathbb{ R}^+
\ \text{\rm and}\ t \in \mathbb{ R} \big\}$,
graphed as $u = x e^{\frac{y}{x}}$,
having symmetries 
$x \partial_x + y \partial_y + u \partial_u$, 
$x \partial_y + u \partial_u$;

\smallskip\noindent{\bf (2c)}\,
$\big\{ r ( 1, e^t, e^{ \theta t}) \in \mathbb{ R}^3 : \, r \in
\mathbb{ R}^+ \ \text{\rm and}\ t \in \mathbb{ R}
\big\}$ with $\theta > 2$
arbitrary, graphed as $u = x \big( \frac{y}{x} \big)^\theta$,
having symmetries
$x \partial_x - (\theta-1) u \partial_u$,
$y \partial_y + \theta u \partial_u$;

\smallskip\noindent{\bf (3)}\,
$\big\{ c ( t) + r c' ( t) \in \mathbb{ R}^3 : \, r \in \mathbb{ R}^+
\ \text{\rm and} \ t \in \mathbb{ R} \big\}$, where $c ( t) := (t,
t^2, t^3)$ parametrizes the {\em twisted cubic} $\{ (t, t^2, t^3): \,
t \in \mathbb{ R}\}$ in $\mathbb{ R}^3$ and $c'(t) = (1, 2t, 3t^2)$,
graphed as $u = -2x^3 + 3xy - 2(x^2-y)^{3/2}$, having symmetries
$x\partial_x + 2y \partial_y + 3u \partial_u$, $\partial_x + 2x
\partial_y + 3y \partial_u$.

\medskip

Conversely, from any parabolic surface $S^2 \subset \R^3$ graphed as
$z = F(x, y)$, one can introduce the 
{\sl tube} para-CR submanifold of solutions
$S^2 \times \R^3$ 
defined by $z + c = F (x+a, y+b)$, differentiate $z_x = F_x$, $z_{xx}
= F_{xx}$, solve for $(a, b, c)$ in these three equations, replace in
$z_y = F_y$, $z_{xxx} = F_{xxx}$, and get a PDE system of the kind
studied in the preceding sections.

Such submanifolds of solutions $z + c = F (x+a, y+b)$ being invariant
under translations along the parameter directions, one gets a {\em
homogeneous} PDE five-variables para-CR structure as soon as the
surface $z = F(x,y)$ is affinely homogeneous.
This observation applies to all cases 
{\bf (1)}, {\bf (2a)}, {\bf (2b)}, {\bf (2c)}, {\bf (3)}.

We leave as an exercise to verify that the PDE systems
{\bf (i)}, {\bf (ii)}, {\bf (iiia)}, {\bf (iiib)}
shown in
Theorem~{\ref{theorem-all-homogeneous-models}}
come from 
{\bf (1)}, {\bf (2a)}, {\bf (2b)}, {\bf (2c)}, {\bf (3)}
by this process, though rearranged differently.


In fact, irrespectively if we knew or not something about
classification of homogeneous parabolic
surfaces, Cartan's method have brought
us only two kinds Maurer-Cartan equations for homogeneous PDE five
variables para-CR structures, namely~({\ref{homo1expl}})
and~({\ref{homo2expl}}).  Thus, cases {\bf (i)}, {\bf (ii)}, {\bf
(iiia)}, {\bf (iiib)} were in fact 
{\em nicely unified in a smaller
number of cases in our first classification}, which is an advantage of
Cartan's reduction with respect to classical classification
results. In conclusion, Cartan's reduction is more natural, because it
groups a scattered number of homogeneous models into only one
$1$-parameter family.

\end{document}